\documentclass[10pt]{article}
\usepackage{geometry}                
\usepackage{amssymb}
\usepackage{epstopdf}
\usepackage{array,amsfonts,amsmath,amsthm,amssymb,graphicx,relsize,url,array,indentfirst,booktabs,geometry,}
\usepackage{listings}
\usepackage{algorithmic}
\usepackage[ruled]{algorithm2e}	
\usepackage{verbatim}
\usepackage{subcaption}

\theoremstyle{plain}
\newtheorem{thm}{Theorem}[section]

\newtheorem{prop}[thm]{Proposition}
\newtheorem{defn}[thm]{Definition}

\newcommand{\norm}[1]{\left\lVert#1\right\rVert}

\newcommand{\tabincell}[2]{\begin{tabular}{@{}#1@{}}#2\end{tabular}}

\def\<{{\langle}}
\def\>{{\rangle}}

\def\bR{{\mathbb{R}}}

\def\H{{\mathcal{H}}}

\def\F{{\mathcal{F}}}

\def\cL{{\mathcal{L}}}
\def\cR{{\mathcal{R}}}

\newcommand{\bs}[1]{\boldsymbol{#1}}
\newcommand{\bit}{\begin{itemize}}
\newcommand{\eit}{\end{itemize}}

\newcommand{\istate}{u}

\newcommand{\ipar}{m}

\newcommand{\ParSpace}{\mathcal{H}}

\newcommand{\wall}{W_{\text{all}}}

\newcommand{\eigen}{\lambda}

\newcommand{\mupr}{\mu_{\text{pr}}}
\newcommand{\mupost}{\mu_{\text{post}}}

\newcommand{\Cpr}{\mathcal{C}_{\text{pr}}}
\newcommand{\Cpost}{\mathcal{C}_{\text{post}}}
\newcommand{\prmean}{\ipar_{\text{pr}}}
\newcommand{\map}{\ipar_{\text{map}}}
\newcommand{\bmap}{\bipar_{\text{map}}}

\newcommand{\Normal}{\mathcal{N}}

\newcommand{\Hmisfit}{\mathcal{H}_m}
\newcommand{\pHmisfit}{\widetilde{\mathcal{H}}_m}

\newcommand{\pilike}{\pi_{\text{like}}}

\newcommand{\Gnoise}{\mathbf{\Gamma}_{\text{n}}}

\newcommand{\Bop}{\mathcal{B}}
\newcommand{\obs}{\mathbf{y}}
\newcommand{\obsd}{n_y}
\newcommand{\ObsSpace}{\mathcal{Y}}
\newcommand{\obsn}{\varepsilon}

\newcommand{\trace}[1]{\operatorname{tr} ({#1}) }
\newcommand{\logdet}[1]{\operatorname{logdet}\left( {#1} \right)}
\newcommand{\ave}{\mathbb{E}}

\newcommand{\bipar}{\mathbf{\ipar}}
\newcommand{\bF}{\mathbf{F}}
\newcommand{\bJ}{\mathbf{J}}
\newcommand{\bmupost}{\mathbf{\ipar}_{\text{post}}}
\newcommand{\bmumap}{\mathbf{\ipar}_{\text{map}}}
\newcommand{\bpostcov}{\mathbf{\Gamma}_{\text{post}}}
\newcommand{\bmupr}{\mathbf{m}_{\text{pr}}}
\newcommand{\bCpr}{\mathbf{\Gamma}_{\text{pr}}}
\newcommand{\bHmisfit}{\mathbf{H}_m}
\newcommand{\bpHmisfit}{\widetilde{\mathbf{H}}_m}

\def\mpr{{m_{\text{pr}}}}

\usepackage{xcolor}

\usepackage{textcomp,fancyhdr,lastpage}
\usepackage{extarrows}
\usepackage[utf8]{inputenc}
\DeclareMathOperator*{\argmax}{arg\,max}

\title{A fast and scalable computational framework for large-scale and high-dimensional Bayesian optimal experimental design}
\author{Keyi Wu, Peng Chen, Omar Ghattas}
\date{}                                           

\begin{document}

\maketitle

\begin{abstract}
We develop a fast and scalable computational framework to solve large-scale and high-dimensional Bayesian optimal experimental design problems. In particular, we consider the problem of optimal observation sensor placement for Bayesian inference of high-dimensional parameters governed by partial differential equations (PDEs), which is formulated as an optimization problem that seeks to maximize an expected information gain (EIG). Such optimization problems are particularly challenging due to the curse of dimensionality for high-dimensional parameters and the expensive solution of large-scale PDEs. To address these challenges, we exploit two essential properties of such problems: (1) the low-rank structure of the Jacobian of the parameter-to-observable map to extract the intrinsically low-dimensional data-informed subspace, and (2) the high correlation of the approximate EIGs by a series of approximations to reduce the number of PDE solves. Based on these properties, we propose an efficient offline-online decomposition for the optimization problem:  an offline stage of computing all the quantities that require a limited number of PDE solves independent of parameter and data dimensions, which is the dominant cost, and an online stage of optimizing sensor placement that does not require any PDE solve. For the online optimization, we propose a swapping greedy algorithm that first construct an initial set of sensors using leverage scores and then swap the chosen sensors with other candidates until certain convergence criteria are met. We demonstrate the efficiency and scalability of the proposed computational framework by a linear inverse problem of inferring the initial condition for an advection-diffusion equation, and a nonlinear inverse problem of inferring the diffusion coefficient of a log-normal diffusion equation, with both the parameter and data dimensions ranging from a few tens to a few thousands. 

\end{abstract}

\section{Introduction}

In mathematical modeling and computational simulation in many scientific and engineering fields, uncertainties are ubiquitous, which may arise from model coefficients, initial or boundary conditions, external loads, computational geometries, etc. It is crucial to quantify and reduce such uncertainties for more accurate and reliable computational prediction and model-based system optimization. Bayesian inference provides an optimal framework for quantifying the uncertainties with suitable prior probability distribution based on domain knowledge or expert belief, and reducing the uncertainties characterized by their posterior probability distribution through fusing noisy experimental or observational data with the model using Bayes' rule. However, it is challenging to acquire enough data if the experiment is expensive or time consuming. In this situation, only a limited number of data can be acquired given budget or time constraint. How to design the experiment such that the limited data can reduce the uncertainties as much as possible becomes very important, which is the central task of optimal experimental design (OED) \cite{AtkinsonDonev92,AlexanderianPetraStadlerEtAl14, AlexanderianGloorGhattas16,LongMotamedTempone15,LongScavinoTemponeEtAl13,AlexanderianPetraStadlerEtAl16, SaibabaAlexanderianIpsen17,CrestelAlexanderianStadlerEtAl17,AttiaAlexanderianSaibaba18,HuanMarzouk13,HuanMarzouk14,HuanMarzouk16, Aretz-NellesenGreplVeroy19, Aretz-NellesenChenGreplEtAl20}.

OED problems can be generally formulated as minimizing or maximizing certain criterion that represents the uncertainty, e.g., the trace (A-optimality) or determinant (D-optimality) of the posterior covariance. In the Bayesian framework, a common choice is an expected information gain (EIG), where the information gain is measured by a Kullback--Leibler divergence between the posterior distribution and the prior distribution, and the expected value is taken as an average of this measure at all realizations of the data. Consequently, two integrals are involved in evaluating the EIG, one with respect to the posterior distribution and the other with respect to the data distribution. Several computational challenges are faced to evaluate the EIG, which includes: (1) evaluation of the double integrals with one involving integration with respect to the posterior distribution may require a large number of samples from the posterior distribution, especially if the uncertain parameters are high-dimensional; (2) the parameter-to-observable map at each sample is expensive to evaluate so only a limited number of map evaluation can be afforded, which is often the case when the map evaluation involves solution of large-scale models, e.g., represented by partial differential equations. For example, in underground fluid flow one needs to infer the infinite-dimensional permeability field from observation of the fluid velocity or pressure at some locations, or in contaminant diffusion and transportation one needs to infer its source or infinite-dimensional initial concentration field from observation of the concentration at certain locations and time instances. Both of the examples feature large-scale models with high-dimensional parameters after (high-fidelity) discretization. Moreover, when the space for the possible experimental design is also high-dimensional, e.g., the number of candidate sensor locations for the observation is high, one faces the challenge of solving high-dimensional optimization problems, i.e., maximizing the EIG with respect to the high-dimensional design variable.

\textbf{Related Work.} OED problems with the above computational challenges have attracted increasing attention in recent years. An infinite-dimensional version of Kullback-Leibler divergence involved in the EIG is studied in \cite{AlexanderianGloorGhattas16}. For linear problems, i.e., the parameter-to-observation map is linear in the parameter, the EIG is equivalent to the D-optimal design, which measures the log-determinant of the prior-covariance preconditioned misfit Hessian, whose computation depends on its dominant eigenvalues. The fast decay of the eigenvalues has been proven for some model problems and numerically demonstrated for many others \cite{
	Bui-ThanhGhattas12a,
	Bui-ThanhGhattasMartinEtAl13,
	AlexanderianPetraStadlerEtAl16, AlexanderianPetraStadlerEtAl17,
	AlexanderianPetraStadlerEtAl14, CrestelAlexanderianStadlerEtAl17,
	PetraMartinStadlerEtAl14, IsaacPetraStadlerEtAl15, SaibabaAlexanderianIpsen17, ChenVillaGhattas17, AlexanderianSaibaba18, ChenVillaGhattas19, ChenGhattas19a, ChenWuChenEtAl19a, ChenGhattas20}. 
Based on this property, efficient methods have been developed to evaluate D-optimal criterion for infinite-dimensional Bayesian linear problems \cite{SaibabaAlexanderianIpsen17, SaibabaKitanidis15}. 
In \cite{AlexanderianSaibaba18}, the authors exploited this property and proposed a gradient-based optimization method for D-optimal experimental design, which was extended in \cite{AttiaAlexanderianSaibaba18} for goal-oriented optimal design of experiments.
For nonlinear problems, a direct statistical estimator of EIG is double-loop Monte Carlo(DLMC) \cite{BeckDiaEspathEtAl18}, which approximates the EIG via inner and outer Monte Carlo sample average approximations (double loops). Both loops need to generate a large number of samples requiring multiple PDE solves for each sample. The authors of
\cite{LongMotamedTempone15, LongScavinoTemponeEtAl13} proposed to approximate the EIG with Laplace approximations, which involve numerous inverse problem solves and eigenvalue decomposition of prior-preconditioned misfit Hessian. The optimal design is obtained by an exhaustive search over a prespecified set of experimental scenarios. Polynomial chaos expansion was employed in \cite{HuanMarzouk13,HuanMarzouk14,HuanMarzouk16} as a surrogate for the expensive PDE model,  which is however not suitable for solving OED problems with high-dimensional parameters. The authors of \cite{KhodjaPrangeDjikpesse10}  introduced a Bayesian methodology for optimal sensor placement with D-optimal criterion for quasi-linear problems and used a greedy algorithm to sequentially select observation locations. \cite{YuZavalaAnitescu18} presented a scalable framework for sensor placement for PDE problems with Laplace approximation and investigated two criteria (total flow variance and A-optimal design) with a sparsity-inducing approach and a sum-up rounding approach to find the optimal design. Both approaches first relax the binary constraints to continuous function constraints to solve the easier continuous optimization problem, and then induce the integer solution.
\cite{ManoharBruntonKutzEtAl18}  considered problems of sensor placement for signal reconstruction in a scalable framework with D-optimal criterion and compared the accuracy and efficiency between convex optimization and QR pivoting (greedy method) to find the optimal design.
The author of \cite{Papadimitriou04} proposed a formulation of optimal sensor placement problem with Laplace method to approximate the expected information gain, and introduced a sequential sensor searching algorithm to find the optimal sensors. However, in all these papers, many expensive PDEs have to be solved in each optimization iteration, either the greedy algorithms to sequentially add sensors or the optimization algorithms to simultaneously update sensors, which may lead to prohibitively large number of expensive PDEs to solve, especially when the parameter dimension or the data dimension is big. 



\textbf{Contributions}. To address the computational challenges for Bayesian OED problems of maximizing EIG with large-scale models and high-dimensional parameters,
we propose a fast and scalable computational framework, fast in that only a limited number of the large-scale models are solved, scalable in that the computational complexity is independent of both the parameter dimension and the data dimension. These advantages are made possible by (1) using a sequence of approximation of the posterior including Laplace approximation, low-rank approximation of the posterior covariance, replacement of the design dependent MAP point for the Laplace approximation by a fixed MAP point and a prior sample point, (2) exploiting an efficient offline-online decomposition of the computation, offline solving all the large-scale models and online solving the model-independent optimization problem, (3) proposing a swapping greedy algorithm used in the online stage to find the optimal design whose computational complexity depends only on the dimension of the subspace informed by the data, not on the nominal parameter and data dimensions. More specifically, for linear problems where maximizing the EIG is equivalent to minimizing the D-optimality, we derive a new approximate form of the D-optimal criterion for which each evaluation of different designs does not involve any PDE solves. We introduce a swapping greedy algorithm that exploit the dominant data subspace information quantified by the Jacobian of the parameter-to-observable map to search for an optimal design. For nonlinear problems, we use Laplace approximation and exploit the high correlation of the approximate EIGs with the Laplace approximation centered at different points. Moreover, we propose an efficient offline-online decomposition of the optimization problem where the key information is extracted in the offline stage by solving a limited number of PDEs, which is then used in the online stage to find the optimal design by the swapping greedy algorithm.
We demonstrate the effectiveness and scalability of our computational framework by two numerical experiments, a linear inverse problem of inferring the initial condition for an advection-diffusion equation, and a nonlinear inverse problem of inferring the diffusion coefficient of a log-normal diffusion equation, with both the parameter and data dimensions ranging from a few tens to a few thousands. 

Paper overview. In Section \ref{sec:background}, we review infinite-dimensional Bayesian inverse problem with 
finite-element discretization. We also review the expected information gain optimal criterion, including the specific formulation for sensor placement problems. In Section \ref{sec:linear}, we present the optimal experimental design problem for Bayesian linear inverse problems and its discretization to finite dimensions. We propose an efficient approach to evaluate the criterion and introduce a swapping greedy algorithm to search for an optimal design. In Section \ref{sec:nonlinear}, we introduce the optimal experimental design problem for nonlinear model with Laplace approximation and its discretization to finite dimension. We formulate a new framework with several further approximations, and extend our greedy algorithm into this framework.  Section \ref{sec:numerical} presents our numerical results for both linear and nonlinear Bayesian inverse problems, followed by the last Section \ref{sec:conclusion} for conclusions.

\section{Bayesian optimal experimental design}
\label{sec:background}

In this section, we present a general formulation Bayesian inverse problem for an abstract forward model and an infinite-dimensional parameter field, which is discretized by finite element approximation to be finite dimensional. We present the optimal experimental design problem as an optimal selection of the sensor locations among candidate sensor locations, based on a commonly used design criterion---expected information gain.
\subsection{Bayesian inverse problem}
\label{sec:bayes}
We consider a forward model presented in an abstract form as 
\begin{equation}\label{eq:forward}
\cR(\istate,\ipar) = 0 \quad \text{ in } \mathcal{V}',
\end{equation}
where $\istate$ is the state variable in a separable Banach space $\mathcal{V}$ with its dual $\mathcal{V}'$ defined in a physical domain $\mathcal{D} \subset \bR^{n_x}$ with Lipschitz boundary $\partial \mathcal{D}$, where ${n_x} = 1, 2, 3$; $\ipar$ is the parameter field to be  inferred, which is assumed to live in a Hilbert space $\mathcal{M}$ defined in $\mathcal{D}$; $\cR(\cdot, \cdot): \mathcal{V} \times \mathcal{M} \to \mathcal{V}'$ represents the strong form of the forward model, whose weak form can be written as: find $u \in \mathcal{V}$ such that 
\begin{equation}\label{eq:forward-weak}
r(u, m, v): = \; _{\mathcal{V}} \langle v, \cR(\istate,\ipar) \rangle_{\mathcal{V}'} = 0 \quad \forall v \in \mathcal{V},
\end{equation} 
where $ _{\mathcal{V}} \langle \cdot, \cdot \rangle_{\mathcal{V}'}$ denotes the duality pairing.

The data model with additive Gaussian noise is given as 
\begin{equation}
\obs = \Bop (\istate) + \mathbf{\epsilon}, 
\end{equation}
where $\Bop: \mathcal{V} \to \bR^{\obsd}$ is an observation operator that maps the state variable to the data $\obs \in \mathcal{Y}  = \bR^{\obsd}$ at $\obsd$ observation points, corrupted with an additive Gaussian noise $\mathbf{\epsilon} \sim \mathcal{N}(0,\Gnoise)$ with symmetric positive definite covariance  matrix $\Gnoise \in \bR^{\obsd \times \obsd}$. For notational convenience, we denote the parameter-to-observable map $\mathcal{F}$ as
\begin{eqnarray}\label{eq:p2o}
\mathcal{F}(\ipar) = \Bop (\istate(\ipar)).
\end{eqnarray}

We consider that $m$ has a Gaussian prior measure $ \mupr = \mathcal{N}(\mpr,\Cpr)$ with mean $\mpr \in \mathcal{M}$ and covariance operator $\Cpr: \mathcal{M} \to \mathcal{M}'$ from $\mathcal{M}$ to its dual $\mathcal{M}'$, which is a strictly positive self-adjoint operator of trace class. In particular, we use $\Cpr = \mathcal{A}^{-\alpha}$ for sufficiently large $\alpha > 0$, such that $2\alpha > {n_x}$, to guarantee that $\Cpr$ is of trace class, where $\mathcal{A}$ is an elliptic differential operator equipped with homogeneous Neumann boundary condition along the boundary $\partial D$ \cite{PetraMartinStadlerEtAl14}. This choice of prior guarantees a bounded pointwise variance and a well-posed infinite-dimensional Bayesian inverse problem.

Under the assumption of Gaussian noise for $\epsilon$, the likelihood function $\pilike(\obs|\ipar)$ satisfies
\begin{equation}\label{eq:likelihood}
\pilike(\obs|\ipar)  \propto \exp(-\Phi(\ipar,\obs)), 
\end{equation}
where the potential $ \Phi(\ipar,\obs)$ is defined as
\begin{equation}\label{eq:potential}
 \Phi(\ipar,\obs) := \frac{1}{2} \norm{\mathcal{F}(\ipar) - \obs}^2_{\Gnoise^{-1}},
\end{equation}
where $||\mathbf{v}||_{\Gnoise^{-1}}^2 = \mathbf{v}^T \Gnoise^{-1} \mathbf{v}$ for any $mathbf{v} \in \bR^{\obsd}$.
By Bayes' rule, the posterior measure $ \mupost (\ipar | \obs)$ of the parameter $m$ conditioned on the observation data $y$ is given by the Radon-Nikodym derivative as 
\begin{equation}
\frac{d \mupost (\ipar | \obs) }{d \mupr(\ipar)} = \frac{1}{Z} \pilike(\obs|\ipar),
\end{equation}
where $Z$ is a normalization constant given by 
\begin{equation}
Z = \int_\mathcal{M} \pilike(\obs|\ipar) d \mupr(\ipar).
\end{equation}


\subsection{Discretization of the Bayesian inverse problem}
\label{sec:discretization}
The parameter field $m \in \mathcal{M}$ in the Hilbert space $\mathcal{M}$ is infinite-dimensional. For computational purpose, we use a finite element discretization to approximate it in a subspace $\mathcal{M}_n \subset \mathcal{M}$ of dimension $n$, which is spanned by piecewise continuous Lagrange polynomial basis functions $\{ \phi_j\}^n_{j=1}$ defined over a mesh with elements of size $h$ and 
vertices $\{x_j\}_{j=1}^n$, such that $\phi_j(x_i) = \delta_{ij}$, $i,j = 1, \dots, n$, where $\delta_{ij}$ denote the Kronecker delta. The approximation of the parameter $m \in \mathcal{M}$ in $\mathcal{M}_n$, denoted as $m_h$, can be expressed as 
\begin{equation}
m_h = \sum^n_{j=1} m_j \phi_j.
\end{equation}
Here
we denote $\bipar = (m_1,\dots, m_n)^T \in \bR^n$  as the coefficient vector of $m_h$. 

Let $\mathbf{M}$ denote the finite element mass matrix whose entries are given by
\begin{equation}
\mathbf{M}_{ij} = \int_D \phi_i(x) \phi_j (x) dx, ~i,j = 1,\dots,n.
\end{equation}
Let $\mathbf{A}$ denote the finite element matrix corresponding to the elliptic differential operator $\mathcal{A}$, i.e., 
\begin{equation}
\mathbf{A}_{ij} = \int_D \, _{\mathcal{M}'} \langle \mathcal{A} \phi_j(x), \phi_i(x) \rangle_\mathcal{M} dx, \quad i,j = 1, \dots, n.
\end{equation}
With the specification of the parameter $\alpha = 2$, which satisfies $2\alpha > d$ for $d\leq 3$, we obtain a discrete covariance matrix $\bCpr$ corresponding to the covariance operator $\Cpr$, given by
\begin{eqnarray}\label{eq:bCpr}
\bCpr = \mathbf{A}^{-1} \mathbf{M} \mathbf{A}^{-1}.
\end{eqnarray}
Then, the prior density for the coefficient vector $\bipar$, which we also call discrete parameter, is given by 
\begin{equation}
\pi_{\text{pr}} (\bipar) \propto \exp\left(-\frac{1}{2} ||\bipar - \bipar_{\text{pr}}||^2_{\bCpr^{-1}} \right),
\end{equation}
where $\bipar_{\text{pr}} \in \bR^n$ is the coefficient vector of the approximation of the prior mean $\prmean$ in the finite element space $\mathcal{M}_n$. Correspondingly, the posterior density of the discrete parameter $\bipar$ follows the Bayes' rule 
\begin{equation}
\pi_{\text{post}} (\bipar | \obs) = \frac{1}{Z} \pilike(\obs|\bipar)  \pi_{\text{pr}}(\bipar),
\end{equation}
where $\pilike(\obs|\bipar)$
is
the likelihood function for the discrete parameter $\bipar$ given by 
\begin{eqnarray}
\pilike(\obs|\bipar) \propto \exp\left( - \Phi(\mathbf{m}, \obs)\right),
\end{eqnarray}
with the potential 
\begin{equation}
\Phi(\mathbf{m}, y) = \frac{1}{2} ||\mathbf{F}(\mathbf{m}) - \obs||^2_{\Gnoise^{-1}},
\end{equation}
where $\bF: \bR^{n} \to \bR^{\obsd}$ denotes the discrete parameter-to-observable map corresponding to \eqref{eq:p2o}, 
which involves the solution of the forward model \eqref{eq:forward-weak} by a finite element discretization in a subspace $\mathcal{V}_{n_u} \subset \mathcal{V}$ spanned by basis functions $\{\psi_j\}_{j=1}^{n_u}$, in which, the finite element state $u_h = \sum_{j=1}^{n_u} u_j \psi_j$ with coefficient vector $\bs{u} = (u_1, \dots, u_{n_u})^T$.



\subsection{Expected information gain} 
To measure the information gained from the observation data, different information criteria have been used, e.g., A-optimal or D-optimal criterion, which uses the trace or determinant of the covariance of the posterior \cite{AlexanderianGloorGhattas15}. Here we choose to use an expected information gain (EIG), which is the Kullback--Leibler (KL) divergence between the posterior and the prior averaged over all data realizations. The KL divergence, denoted as $D_{\text{KL}}$, also known as the information gain, is used to measure the information gained from data $\obs$, which is defined as
\begin{equation}\label{eq:KL}
D_{\text{KL}}(\mupost (\cdot | \obs) \| \mupr ) := \int_{\ParSpace} \ln\left(
\frac{d\mupost(\ipar | \obs)}{d\mupr(\ipar)}\right) d\mupost(\ipar | \obs).
\end{equation}

The EIG, denoted as $\Psi$, takes all possible realizations of the data $\obs \in \mathcal{Y}$ into account, which is defined as
\begin{equation}\label{eq:EIG}
\begin{split}
\Psi 
&:= \ave_\obs \left[D_{\text{KL}}(\mupost(\cdot |
                 \obs)\| \mupr)\right] \\  
& = \int_{\ObsSpace} D_{\text{KL}}(\mupost(\cdot | \obs) \| \mupr )
       \, \pi(\obs)\, d\obs \\ 
& =\int_{\ObsSpace}  \int_{\ParSpace} D_{\text{KL}}(\mupost(\cdot | \obs)
\| \mupr ) \, \pilike(\obs | \ipar) \, d\mupr(\ipar)\, d\obs \\
& = \int_{\ParSpace}\int_{\ObsSpace} D_{\text{KL}}(\mupost(\cdot | \obs)
        \| \mupr ) \, \pilike(\obs | \ipar) \, d\obs \, d\mupr(\ipar),
\end{split}
\end{equation}
where $\pi(\obs)$ in the second equality is the density of the data $\obs$, which follows a Gaussian distribution $\mathcal{N}(\mathcal{F}(m), \Gnoise)$ conditioned on the parameter $m$, i.e., $\pi(\obs) = \pilike(\obs | \ipar) \, d\mupr(\ipar)$ which is used in the third equality. The fourth equality is obtained by switching the order of integration under the assumption that $D_{\text{KL}}(\mupost(\cdot | \obs)
\| \mupr )$ is integrable by Fubini theorem. 
An efficient evaluation of the KL divergence in is presented in Sections \ref{sec:linear} and \ref{sec:nonlinear} for linear and nonlinear Bayesian inverse problems, respectively.

\subsection{Optimal experimental design for sensor placement} 
\label{sensor}
We consider a sensor placement problem in which we have a set of candidate sensor locations $\{x_i \}_{i=1}^d$ in the physical domain $\mathcal{D}$, among which we need to choose $r$ locations (due to limited budget) to place the sensors for data collection. To represent the selection, we use a design matrix $W$, which is a boolean matrix $W \in \bR^{r\times d}$ such that $W_{ij} = 1$ if the $i$-th sensor is placed at the $j$-th candidate sensor location, i.e., 
\begin{equation}\label{eq:W}
W \in \mathcal{W} := \left\{W \in \bR^{r\times d}: W_{ij} \in \{0,1 \},\;  \sum^d_{j=1} W_{ij}=1, \; \sum^r_{i=1} W_{ij} \in  \{0,1 \}  \right\}.
\end{equation}
We consider the case of uncorrelated observation noises with covariance matrix 
\begin{equation}\label{eq:Gamma_n_d}
\Gnoise^d = \text{diag}(\sigma_1^2,\dots, \sigma_d^2)
\end{equation}
at $d$ candidate sensor locations, where $\sigma^2_j$ indicates the noise level at the $j$-th candidate sensor location.
By $\mathcal{B}_d$ we denote the observation operator at all $d$ candidate sensor locations. For a specific design $W$, we have the observation operator $\mathcal{B} = W\mathcal{B} _d$, so that the EIG $\Psi(W)$ depends on the design matrix through the observation operator $\mathcal{B}$. To this end, 
the optimal experimental design problem can be formulated as: find an optimal design matrix $W \in \mathcal{W}$ such that 
\begin{equation}
W = \argmax_{W \in  \mathcal{W} }\Psi(W).
\end{equation}

\section{Linear Bayesian inverse problems}
\label{sec:linear}
We consider linear Bayesian inverse problems in this section, where the parameter-to-observable map $\mathcal{F}(m)$ is linear with respect to the parameter $m$. In the following, we present an explicit form of the EIG for the linear inverse problem, derive an optimization problem that can be divided into an ``offline" phase of computing the expensive part that involves solution of the forward model and an ``online" phase of optimization with respect to the design matrix, which is free from solving the forward model, and introduce two greedy algorithms for the online phase of optimization.

\subsection{Expected information gain}
Given a Gaussian prior measure $\mupr = \Normal\left( \prmean, \Cpr \right)$, for a linear parameter-to-observable map $\F$, the posterior measure $\mupost $ is also Gaussian, $\mupost = \Normal\left(m_{\text{post}}, \Cpost \right)$, with
\begin{equation}
\Cpost = (\mathcal{F}^* \Gnoise^{-1} \mathcal{F} + \Cpr^{-1} )^{-1}, \quad m_{\text{post}} = \Cpost (\mathcal{F}^* \Gnoise^{-1} \obs + \Cpr^{-1} m_{\text{pr}}),
\end{equation}
where $\mathcal{F}^*$ is the adjoint of $\mathcal{F}$.
 For the Gaussian prior and posterior,  we have a closed form for the KL divergence defined in \eqref{eq:KL} as \cite{AlexanderianGloorGhattas16} 
 \begin{equation}\label{eq:KL-linear}
D_{\text{KL}}( \mupost \| \mupr ) =
\frac{1}{2}\left[\logdet{\mathcal{I} + \pHmisfit} 
  - \trace{\Cpost^\frac{1}{2} \Hmisfit \Cpost^\frac12}
  + \| m_{\text{post}} -m_{\text{pr}} \|^2_{\Cpr^{-1}}\right].
\end{equation}
Here $ \pHmisfit = \Cpr^\frac{1}{2} \, \Hmisfit \, \Cpr^\frac12 $ is a
prior-preconditioned data misfit Hessian, and $\Hmisfit = \mathcal{F}^*  \Gnoise^{-1}  \mathcal{F} $ is the Hessian of the potential $\Phi(m,\obs)$. The term $\logdet{\mathcal{I} + \pHmisfit}$ is the logarithm of the determinant of $\mathcal{I} + \pHmisfit$, and $\trace{\Cpost^\frac{1}{2} \Hmisfit \Cpost^\frac12}$ is the trace of $\Cpost^\frac{1}{2} \Hmisfit \Cpost^\frac12$.
Moreover, it is shown in \cite{AlexanderianGloorGhattas15} that the EIG is simply
\begin{equation}\label{eq:EIGlinear}
\Psi = \frac{1}{2}\logdet{\mathcal{I} + \pHmisfit}.
\end{equation}

%
By the discretization in Section \ref{sec:discretization}, we obtain the discrete Gaussian posterior measure $\mathcal{N}(\bmupost, \bpostcov)$ with covariance and mean
\begin{equation}
\bpostcov = (\bF^T\Gnoise^{-1} \bF+ \bCpr^{-1})^{-1}, \quad \bmupost = \bpostcov(\bF^T \Gnoise^{-1} \obs + \bCpr^{-1} \bmupr).
\end{equation}
where $\bF^T \in \bR^{\obsd \times n}$ is the discrete linear parameter-to-observable map and $\bF^T$ is its transpose.
Let $\bHmisfit = \bF^T\Gnoise^{-1} \bF $, and $\bpHmisfit = \bCpr^{\frac{1}{2}}\bHmisfit \bCpr^{\frac{1}{2}}$ with the discrete prior covariance $\bCpr$, 
we obtain a discrete EIG corresponding to \eqref{eq:EIGlinear} as 
\begin{equation}\label{eq:EIGdiscrete}
\Psi = \frac{1}{2}\logdet{\mathbf{I} + \bpHmisfit}.
\end{equation}

When the dimension of the discrete parameter $\mathbf{m}$ becomes large,
and the discrete parameter-to-observable map $\mathbf{F}$ involves the inverse of some large-scale discrete differential operator, it is computationally prohibitive to compute $\bpHmisfit$ in \eqref{eq:EIGdiscrete}. 
However, by the blessing of the common ill-posedness of high-dimensional inverse problems, i.e., the data only effectively infer a low-dimensional subspace of the high-dimensional parameter space, which is depicted by a fast decay of the eigenvalues of the Hessian $\bpHmisfit$,
 we can compute a low-rank approximation of the Hessian (with certain transformation) as detailed in next section. 
 

\subsection{Low rank approximation}
For the design problem defined in section \ref{sensor}, we denote $\bF_d$ as the discretized parameter-to-observable map at all the $d$ candidate sensor locations. Then for a specific design $W$, we have
\begin{equation}\label{eq:FFstar}
\bF = W\bF_d \quad \text{ and } \bF^T = \bF_d^T W^T.
\end{equation}
We present an efficient approximation of the discrete EIG given in \eqref{eq:EIGdiscrete} and establish an error estimate for the approximation. To start, we need the following results, which are proven in \cite{AlexanderianSaibaba18} for Proposition \ref{prop:1} and in \cite{Pozrikidis14} for Proposition \ref{prop:WA}.
\begin{prop}\label{prop:1}
Let $A,B \in \mathbb{C}^{n \times n}$ be Hermitian positive semidefinite with $A \geq B$, then
\begin{equation}
0 \leq \log \det (I+A) -\log \det (I+B)  \leq \log \det (I+A-B).
\end{equation}
\end{prop}
\begin{prop}\label{prop:WA}
Let $A$ and $B$ be matrices of size $m \times n$ and $n \times m$ respectively, then
  \begin{equation}\label{eq:WA}
\det (\mathbf{I}_{n\times n} + BA) = \det (\mathbf{I}_{m\times m} + AB), 
  \end{equation}
  which is known as Weinstein-Aronszajn identity.
\end{prop}

\begin{thm}\label{thm:low-rank-EIG}
	Let $\mathbf{H}_d: =\hat{ \bF}_d  \, \bCpr\,  \hat{\bF}_d^T  \in \bR^{d\times d}$ with $\hat{\bF}_d =  (\Gnoise^d)^{-\frac{1}{2}}  \bF_d$, where $\Gnoise^d$ is defined in \eqref{eq:Gamma_n_d}. We compute a low-rank decomposition of $\mathbf{H}_d$ as 
	\begin{equation}\label{eq:Hdhat}
	\hat{\mathbf{H}}_d = U_k \Sigma_k U_k^T,
	\end{equation}
	where $(\Sigma_k, U_k)$ represent the $k$ dominant eigenpairs, with $\Sigma_k = \text{diag}(\lambda_1, \dots, \lambda_k)$ for eigenvalues $\lambda_1 \geq \cdots \geq \lambda_k$. Moreover, let $\hat{ \Psi }$ denote an approximate EIG defined as  
	\begin{equation}\label{eq:Psihat}
	\hat{ \Psi }(W) := \frac{1}{2}\logdet{\mathbf{I}_{r\times r}+  W \hat{\mathbf{H}}_d W^T}.
	\end{equation}
	Then we have
	\begin{equation}\label{eq:boundPsi}
	0 \leq  \Psi(W) -\hat{ \Psi }(W) \leq \frac{1}{2} \sum_{i = k+1}^d \log(1+\lambda_i),
	\end{equation}
	where $\lambda_{k+1}, \dots, \lambda_d$ are the remaining eigenvalues of $\mathbf{H}_d$.
\end{thm}

\begin{proof}
For the definition of $\bpHmisfit = \bCpr^{\frac{1}{2}}\bF^T\Gnoise^{-1}\bF\bCpr^{\frac{1}{2}}$, we first denote $\hat{\bF} = \Gnoise^{-\frac{1}{2}}\bF$. Then using \eqref{eq:FFstar} we have 
\begin{equation}
\hat{\bF} = \Gnoise^{-\frac{1}{2}}W \bF_d = W (\Gnoise^d)^{-\frac{1}{2}} \bF_d = W \hat{\bF}_d.
\end{equation}
 We can write the form of EIG in \eqref{eq:EIGdiscrete} as 
\begin{equation}
\begin{split}
{\Psi}(W)&  = \frac{1}{2}\logdet{\mathbf{I}_{n\times n}+  \bCpr^{\frac{1}{2}}\bF^T\Gnoise^{-1}\bF\bCpr^{\frac{1}{2}} }\\
&= \frac{1}{2}\logdet{\mathbf{I}_{n\times n}+  \bCpr^{\frac{1}{2}}\hat{\bF}^T\hat{\bF}\bCpr^{\frac{1}{2}} }\\
&= \frac{1}{2}\logdet{\mathbf{I}_{n\times n} +  \bCpr^{\frac{1}{2}} \,\hat{\bF}_d^T W^T W\hat{\bF}_d  \, \bCpr^{\frac{1}{2}} },
\end{split}
\end{equation} 
where in the second equality we use the relation \eqref{eq:FFstar}.

By the Weinstein-Aronszajn identity \eqref{eq:WA}, we have
\begin{equation}
{\Psi}(W)  =  \frac{1}{2}\logdet{\mathbf{I}_{r\times r} + W\hat{\bF}_d  \, \bCpr^{\frac{1}{2}} \bCpr^{\frac{1}{2}} \, \hat{\bF}_d^T W^T }
\end{equation}

\begin{equation}
 =  \frac{1}{2}\logdet{\mathbf{I}_{r\times r}  +  W\mathbf{H}_d W^T},
\end{equation}
where in the second equality we use the definition of $\mathbf{H}_d$. We denote the eigenvalue decomposition of $\mathbf{H}_d$ as 
\begin{equation}
\mathbf{H}_d = U_k \Sigma_k U_k^T + U_\perp \Sigma_\perp U_\perp^T,
\end{equation}
where $(\Sigma_k, U_k)$ represent the $k$ dominant eigenpairs, and $(\Sigma_\perp, U_\perp)$ represent the remaining $d - k$ eigenpairs. Then by definition of the EIG approximation $\hat{ \Psi }(W)$ in \eqref{eq:Psihat}, we have 
\begin{equation}
\begin{split}
 \Psi(W) -\hat{ \Psi }(W)  &= 
 \frac{1}{2}  \logdet{\mathbf{I}_{r\times r}  +  W\mathbf{H}_d W^T }- \frac{1}{2}  \logdet{\mathbf{I}_{r\times r}+  W \hat{\mathbf{H}}_d  W^T}\\
 & \leq \frac{1}{2}\logdet{\mathbf{I}_{r\times r}  +  W\mathbf{H}_d W^T - W \hat{\mathbf{H}}_d W^T}\\
 & = \frac{1}{2}\logdet{\mathbf{I}_{r\times r}  +  W U_\perp \Sigma_\perp U_\perp^T W^T}\\
 & = \frac{1}{2}\logdet{\mathbf{I}_{r\times r}  +  W U_\perp \Sigma_\perp^{1/2} \Sigma_\perp^{1/2} U_\perp^T W^T}\\
& =  \frac{1}{2}\logdet{\mathbf{I}_{(d-r)\times (d-r)}  +\Sigma_\perp^{1/2} U_\perp^T W^T W U_\perp \Sigma_\perp^{1/2} },
\end{split}
\end{equation}
 where we use Proposition \ref{prop:1} for the inequality, and Proposition \ref{prop:WA} for the last equality. By definition of the design matrix $W$ in \eqref{eq:W}, we have that
$ W^T W \in \bR^{d \times d}$ is a square diagonal matrix with $r$ diagonal entries as one and the others zero, which satisfies $W^T W \leq\mathbf{I}_{d\times d}$. Consequently, we have 
\begin{equation}
\Psi(W) -\hat{ \Psi }(W) \leq \frac{1}{2} \logdet{\mathbf{I}_{(d-r)\times (d-r)}  +\Sigma_\perp^{1/2} U_\perp^T U_\perp \Sigma_\perp^{1/2}} = \frac{1}{2} \logdet{\mathbf{I}_{(d-r)\times (d-r)}  +\Sigma_\perp}, 
\end{equation}
where we use the orthonormality $U_\perp^T U_\perp = \mathbf{I}_{(d-r)\times (d-r)}$ for the eigenvectors in the equality. This concludes the upper bound for $ \Psi(W) -\hat{ \Psi }(W)$. The lower bound in \eqref{eq:boundPsi} is implied by Proposition \ref{prop:1} and $\mathbf{H}_d \geq \hat{\mathbf{H}}_d$.
\end{proof}

The Hessian $\mathbf{H}_d$ is a large dense matrix which usually has a rapidly-decaying spectrum. The low-rank approximation \eqref{eq:Hdhat} can be efficiently computed by a randomized singular-value-decomposition(SVD) algorithm. It is a flexible and robust method that only requires Hessian action in a limited number ($O(k+p)$) of given directions instead of forming the full Hessian matrix \cite{HalkoMartinssonTropp11}, particularly useful for the large-scale problems, see Algorithm \ref{alg:RSVD}.  

\begin{algorithm}[H]
	\small
	\caption{Randomized SVD to compute \eqref{eq:Hdhat}}
	\label{alg:RSVD}
	\begin{algorithmic}[1]
		\STATE Generate  i.i.d. Gaussian matrix $\bs{\Omega} \in \mathbb
		R^{d \times (k+p)}$ with an oversampling parameter $p$ very small (e.g., $p = 10$). 
		\STATE  Compute $\bs{Y} = \mathbf{H}_d \bs{\Omega} .$ 
		\STATE  Compute the QR factorization $\bs{Y} = \bs{Q}\bs{R}$ satisfying $\bs{Q}^T \bs{Q} = \bs{I}$.
		\STATE Compute $\bs{B} = \bs{Q}^T \mathbf{H}_d  \bs{Q}$. 
		\STATE Solve an eigenvalue problem for $\bs{B}$ such that $\bs{B} = \bs{Z} \bs{\Sigma} \bs{Z}^T$.
		\STATE Form $U_k = \bs{Q}\bs{Z}[1:k]$ and $\Sigma_k =\bs{\Sigma}[1:k, 1:k]$.
	\end{algorithmic}
\end{algorithm}

We can see that it is a matrix-free eigen-solver with both step $2$ and $4$ requires $O(k+p)$ independent Hessian-vector products. Each Hessian-vector product involves a pair of forward and adjoint PDE solves. For many problems, the rank $k$ is independent of the parameter dimension and data dimension. Therefore, the dominant computational cost, i.e., the number of PDE solves is independent of the parameter dimension and data dimension.
%

Once the low-rank approximation \eqref{eq:Hdhat} is computed, the approximate EIG in \eqref{eq:Psihat} does not involve solution of the forward model. In the online stage, we only need to find a design matrix $W$ such that 
\begin{equation}
\label{eq:Wproblem}
W = \argmax_{W \in  \mathcal{W} }\hat{\Psi}(W).
\end{equation}


\subsection{Greedy algorithms}
The problem \eqref{eq:Wproblem} can also be treated as the optimization of set functions under cardinality constraint set as each design matrix $W$ represent a choice of subset of candidate sensor location . The set functions usually exhibit properties such as submodularity(Definition \ref{def:1}) that allow for efficient optimization methods\cite{Fujishige05}. It is well-know that the log-determinant function is a submodular function\cite{HanKambadurParkEtAl17} with the definition of submodularity given as follows.
\begin{defn}
\label{def:1}
Let $f$ be a set function on $\mathcal{V}$, i.e ., $f: 2^{\mathcal{V}} \to \bR$. Then $f$ is submodular if for every $A,B \subset \mathcal{V}$ with $A \subset B$ and every $v \in \mathcal{V}\setminus B$, we have $f(A \cup \{ v\}) - f(A) \geq f(B \cup \{ v\}) - f(B)$.
\end{defn}
Submodularity is a useful property of set functions with deep theoretical results and applications in combinational optimization and machine learning\cite{Bachothers13,KrauseGolovin14,WeiIyerBilmes15,NemhauserWolseyFisher78}. 
 The problem of submodular function maximization is a classic NP-hard problem. A general approach of solving it with cardinality constraints is a greedy algorithm. A celebrated result by \cite{NemhauserWolseyFisher78} proves that greedy algorithms can provide a good approximation to the optimal solution of the NP-hard optimization problem despite its simplicity.\\
 The standard greedy algorithm has theoretical guarantees for maximizing submodular functions.  Recently, there are some investigations on the theoretical support for its performance for non-submodular functions. The greedy algorithm enjoys strong empirical performance and provides near-optimal solutions in practice for many submodular and non-submodular functions as shown in \cite{BianBuhmannKrauseEtAl17, JagalurMohanMarzouk20}. Yet the repeated function evaluations required to compute object function in each greedy step may comprise a large computational cost if each function evaluation is expensive. \\
The following Algorithm~\ref{alg:greedy1} is the standard greedy algorithm that starting with the empty set $S^0$, during each iteration $t$, the element maximizing the function is added to the chosen sensor set $S^t$. We denote it as standard greedy algorithm here.\\
  \begin{algorithm}[H]
\small
   \caption{A standard greedy algorithm}
   \label{alg:greedy1}
\begin{algorithmic}[1]
  \STATE {\bfseries Input}: $U_k, \Sigma_k$ in \eqref{eq:Hdhat}, a set $S$ of $d$ candidate sensors, a budget of $r$ sensors to place, an initial set $S^0 = \emptyset$.
   \FOR{$t = 1,\dots,r$}
   \STATE $v^* \leftarrow \arg \max_{v \in S \setminus S^{t-1}}\hat{ \Psi }(v)$, 
   where $ \hat{ \Psi }(v) := \frac{1}{2}\logdet{\mathbf{I}+  W_v U_k \Sigma_k U_k^T W_v^T }$, 
   \\ for the design matrix $W_v$ corresponding to the sensor set $S^{t-1} \cup \{ v\}$.
   \STATE $S^t \leftarrow S^{t-1} \cup \{ v^*\}$
  \ENDFOR
  \STATE {\bfseries Output:} optimal sensor choice $S^r$
    \end{algorithmic}
\end{algorithm}

We can see with each sensor selection, we apply $W$ on the eigenvector matrix $U_k$ to actually select the rows of $U_k$, which can be treated as a row-selection problem to find a subset of rows that captures $U_k$ as much as possible.  This quantity has a nature interpretation in terms of statistical leverage(Definition \ref{def:2})\cite{PapailiopoulosKyrillidisBoutsidis14} and it has been used extensively to identify ``outlying", or we can say in our problem, more ``informative" data point\cite{ChatterjeeHadi09,JagalurMohanMarzouk20}. If we employ this leverage score as a bias toward more ``informative" rows, which then provide a ``nice" starting point for the selection in greedy algorithm\cite{BoutsidisMahoneyDrineas08}.
\begin{defn}
\label{def:2}
Let $U_k \in \bR^{d \times k}$ contain the eigenvectors corresponding to the $k$ dominant eigenvalues of a $d \times d$ symmetric matrix $A$ with rank$(A)\geq k$.  Then the (rank-$k$) leverage score of the $i$-th row of $A$ is defines as $l^{k}_{i} = \norm{[U_k]_{i,:}}^2$ for $i = 1,\cdots,n$ where $[U_k]_{i,:}$ denotes the $i$-th row of $U_k$.
\end{defn}
Instead of choosing the sensors one-by-one as in Algorithm \ref{alg:greedy1}, we use the leverage score information from $U_k$ as a criterion to select the sensors. We first choose the $r$ rows that have the top-$r$ leverage scores of $U_k$ as the initial sensor set $S^0$.  At each iteration $t =  1, \dots, r$, we swap the $t$-th sensor in $S^t$ with one from $S\setminus S^t$ that maximizes the leverage score. Then we set the resulting sensor set $S^r$ as $S^0$ and repeat the whole process of swapping sensors until it converges that no sensor is changed. We call it swapping greedy algorithm.

  \begin{algorithm}[H]
\small
   \caption{A swapping greedy algorithm}
   \label{alg:greedy2}
\begin{algorithmic}[1]
  \STATE {\bfseries Input}: $U_k, \Sigma_k$ in \eqref{eq:Hdhat}, a set $S$ of $d$ candidate sensors, a budget of $r$ sensors to place.
  \STATE Compute the initial guess sensor set $S^* =\{s_1,\dots,s_r\} \subset S$ based on leverage scores of $U_k$. $S^0 = \{ \emptyset\}$.
  \WHILE{$S^* \neq S^0$}
   \STATE $S^0 \leftarrow S^*$.
   \FOR{$t = 1,\dots,r$}
   \STATE $v^* \leftarrow \arg \max_{v \in \{ s_t\} \cup (S \setminus S^{t-1})}\hat{ \Psi }(W) = \frac{1}{2}\logdet{\mathbf{I}+  W_v U_k \Sigma_k U_k^T W_v^T} $, \\
   $W_v$ is the design matrix for the sensor choice $S^{t-1}  \setminus  \{ s_t\}\cup \{ v\}$.
   \STATE $S^t \leftarrow (S^{t-1}  \setminus  \{ s_t\})\cup \{ v^*\}$.
  \ENDFOR
   \STATE $S^* \leftarrow S^r$.
  \ENDWHILE
  \STATE {\bfseries Output:} optimal sensor choice $S^*$.
    \end{algorithmic}
\end{algorithm}

%

\section{Nonlinear Bayesian inverse problems}
\label{sec:nonlinear}

We consider Bayesian nonlinear inverse problem in this section, where the parameter-to-observable map $\F(m)$ is nonlinear with respect to the parameter. For such problems we do not have the data-independent EIG as in \eqref{eq:EIGlinear}, and need to use a data-averaged KL divergence by a sample average approximation as 
\begin{eqnarray}\label{eq:Phiapprox}
\Psi  \approx \frac{1}{N_s} \sum_{i=1}^{N_s} D_{\text{KL}}(\mupost(\cdot
| \obs_i) \| \mupr ), 
\end{eqnarray}
where the data $\obs_i$ are given by
\begin{equation}\label{eq:obsi}
\obs_i = \mathcal{F}( \ipar_i) + \obsn_i,
\end{equation}
with $N_s$ i.i.d.\ samples $ \ipar_i \sim
\mupr$ and $\obsn_i \sim\mathcal{N}(0,\Gnoise) $, $i = 1, \dots, N_s$. 
 This involves the critical challenges of (1) computation of the posterior for each data realization, (2) an high-dimensional integral for the KL divergence, and, (3) a complex and implicit dependence of the EIG on the design matrix through the solution of the Bayesian nonlinear inverse problem. To tackle these challenges,  beyond the low-rank approximation of the Hessian introduced in last section, we propose a sequence of further approximations including (1) Laplace approximation of the posterior for each data realization, which leads to an efficient computation of the posterior, (2) an approximation of a varying MAP point for each configuration of the sensors by a fixed MAP point obtained with data from all candidate sensors, which prevents solution of the nonlinear inverse problems at each configuration of the sensors, and, (3) an approximation of the fixed MAP point by the synthetic sample drawn from the prior, which further gets rid of finding the MAP point for the nonlinear inverse problem. We demonstrate the accuracy and efficiency of the proposed approximations and employ them in the greedy algorithms for scalable and efficient optimization of the design matrix.

\subsection{Laplace approximation}
Laplace approximation (LA) seeks an approximation of the nonlinear Bayesian posterior $\mupost$ for given data $\obs$ by a Gaussian distribution $ \mupost \approx \mupost^{\text{LA}}$, with the mean given by the MAP point $\map$ as
\begin{equation}\label{eq:Laplace-mean}
\map :=  \operatornamewithlimits{argmin}_{\ipar \in \ParSpace}\tfrac12 \| \F (\ipar) - \obs
 \|^2_{\Gnoise^{-1}} + 
\tfrac{1}{2}\| \ipar - \prmean \|^2_{\Cpr^{-1}},
\end{equation}
and the covariance $\Cpost$ is given by 
\begin{equation}\label{eq:Laplace-cov}
\Cpost^{-1} : =  \Hmisfit(\map) + \Cpr^{-1},
\end{equation}
where $\Hmisfit(\map)$ is the Hessian of the misfit term $\tfrac12 \| \F (\ipar) - \obs
\|^2_{\Gnoise^{-1}}$ w.r.t.\ the parameter $m$, evaluated at the MAP point $\map$.
For efficient computation of the EIG, we employ a Gauss--Newton approximation of the Hessian as (with slight abuse of notation)
\begin{equation}
 \Hmisfit(\map) := \mathcal{J}^*(\map) \, \Gnoise^{-1} \,
\mathcal{J}(\map),
\end{equation}
where $\mathcal{J}(\map)$ is the Jacobian of the parameter-to-observable map $\F (\ipar)$ w.r.t.\ $\ipar$ evaluated at $\map$, i.e., 
\begin{equation}\label{eq:Jacobian}
\mathcal{J}(\map) : = \nabla_m \F (\ipar)|_{m = \map}.
\end{equation}
By the Laplace approximation of the posterior, similar to \eqref{eq:KL-linear}, the analytical form of the KL divergence is given as 
\begin{equation}\label{eq:KL-nonlinear}
D_{\text{KL}}(\mupost^{\text{LA}} \| \mupr ) =
\frac{1}{2} \left[\logdet{\mathcal{I} + \pHmisfit} 
- \trace{\Cpost^\frac{1}{2} \Hmisfit \Cpost^\frac12}
+ \|\map - \prmean \|^2_{\Cpr^{-1}}\right],
\end{equation}  
where $ \pHmisfit = \Cpr^\frac{1}{2} \, \Hmisfit \, \Cpr^\frac12 $.
By a finite dimensional discretization as presented in Section \ref{sec:discretization}, we obtain the discrete Laplace approximation with mean $\bmumap$ as the coefficient vector of $\map$, and the inverse of the covariance matrix
\begin{equation}\label{eq:discrete-PostCov}
\begin{split}
\bpostcov^{-1} & = \bHmisfit(\bmumap)+ \bCpr^{-1} = \bJ^T(\bmumap) \Gnoise^{-1} \bJ(\bmumap) + \bCpr^{-1},
\end{split}
\end{equation}
where $\bJ$ is the discretization of the Jacobian $\mathcal{J}$ in \eqref{eq:Jacobian}, whose efficient computation is deferred to Section \ref{sec:complexity}. Moreover, we have the discrete KL divergence corresponding to \eqref{eq:KL-nonlinear} given by (with slight abuse of notation) 
\begin{equation}\label{eq:KL-NonlinearDiscrete}
D_{\text{KL}}( \mupost^{\text{LA}} \| \mupr ) = \frac{1}{2}\left[\logdet{\mathbf{I} + \bpHmisfit(\bmumap)} 
  - \trace{\bpostcov ^\frac{1}{2} \bHmisfit(\bmumap) \bpostcov ^\frac12}
  + \|\bmumap - \bmupr \|^2_{\bCpr^{-1}}\right],
\end{equation}
where $\bpHmisfit (\bmumap)$ is the prior-preconditioned Hessian given by 
\begin{equation}\label{eq:prior-pre-Hessian}
\bpHmisfit (\bmumap)= \bCpr^{\frac{1}{2}}\bHmisfit (\bmumap)\bCpr^{\frac{1}{2}}.
\end{equation}

\subsection{Low rank approximation}

By the Laplace approximation of the Bayesian posterior at each given data $\obs_i$, $i = 1, \dots, N_s$, we can compute the EIG similar to that for the linear inverse problem by first computing a low rank approximation of the prior-preconditioned Hessian \eqref{eq:prior-pre-Hessian}. More specifically, we need to compute the eigenvalues of $\bpHmisfit$ at the MAP points $\bmumap^i$ for each $i = 1, \dots, N_s$. 
Let $\bJ_d$ denote the Jacobian at all $d$ candidate sensor locations evaluated at $\bmumap^i$,
then we have 
\begin{equation}
\bJ = W \bJ_d \quad \text{ and } \quad  \bJ^T = \bJ^T_d W^T.
\end{equation}
As in Theorem \ref{thm:low-rank-EIG}, by introducing $\hat{\bJ}_d =  (\Gnoise^d)^{-\frac{1}{2}}  \bJ_d$,
we have
\begin{equation}
\label{eq:hessian}
\bpHmisfit = \bCpr^{\frac{1}{2}} \bJ^T_d W^T \Gnoise^{-1} W \bJ_d \bCpr^{\frac{1}{2}} = \bCpr^{\frac{1}{2}} \hat{\bJ}^T_d W^T W \hat{\bJ}_d \bCpr^{\frac{1}{2}}.
\end{equation}
To compute the eigenvalues of $\bpHmisfit$, we use the following result.
\begin{prop}\label{prop:2}
Let $A$ and $B$ be matrices of size $m \times n$ and $n \times m$ respectively, then $AB$ and 
$BA$ have the same non-zero eigenvalues.
\end{prop}
\begin{proof}
	Let $v$ be an eigenvector of $BA$ corresponding to some eigenvalue $\lambda \neq 0$, then 
	\begin{equation}
	(AB)Av = A(BA) v = \lambda A v.
	\end{equation}
	Since $\lambda \neq 0$, $Av$ is an eigenvector of $AB$ corresponding to eigenvalue $
	\lambda$. Thus non-zeros eigenvalues of $BA$ are also eigenvalues of $AB$. Then switching the roles of $A$ and $B$ above, we can see that non-zero eigenvalues of $AB$ are also eigenvalues of $BA$.
\end{proof}
By Proposition \ref{prop:2}, the eigenvalues of $\bpHmisfit$ are the same as the eigenvalues of $W \mathbf{H}_d W^T$, where 
\begin{equation}\label{eq:Hd-low-rank}
\mathbf{H}_d = \hat{\bJ}_d \bCpr \hat{\bJ}^T_d \approx \hat{\mathbf{H}}_d = U_k \Sigma_k U_k^T,
\end{equation}
where the low rank approximation $\hat{\mathbf{H}}_d $ can be efficiently computed by the randomized Algorithm \ref{alg:RSVD}. Therefore, the eigenvalues of $\bpHmisfit$ can be efficiently computed by first forming a small matrix $W \hat{\mathbf{H}}_d  W^T$ of size $r \times r$, and then computing the eigenvalues of $W \hat{\mathbf{H}}_d  W^T$.
 Note that both the MAP point $\bmap^i$ and the Hessian $\bHmisfit$ depend on the design matrix $W$, with the former depending on $W$ through the optimization problem \eqref{eq:Laplace-mean}, and the latter through its definition in \eqref{eq:discrete-PostCov}, where the Jacobian $\bJ$ depends on the design matrix. As a result, the eigenvalues of $\bpHmisfit$ depends on $W$.
With these eigenvalues, we obtain the approximations for the quantities in \eqref{eq:KL-NonlinearDiscrete} as \cite{VillaPetraGhattas18}
\begin{equation}
\logdet{\mathbf{I} + \bpHmisfit(\bmap^i) } \approx \sum_{j=1}^{r}
\ln(1+\lambda_j^i(W)) \quad \text{ and } \quad
\trace{\bpostcov^\frac{1}{2}\bHmisfit(\bmap^i)\bpostcov^{\frac12}}
\approx \sum_{j=1}^{r} \frac{\lambda_j^i(W)}{1+\lambda_j^i(W)},
\end{equation}
which leads to a Laplace approximation (with Gauss--Newton Hessian and low rank approximation) of the EIG \eqref{eq:Phiapprox} as 
\begin{equation}\label{eq:LA-EIG}
\hat{ \Psi }(W) =  \frac{1}{N_s}\sum_{i=1}^{N_s} \frac{1}{2}
\left[ \sum_{j=1}^{r}\left( \ln(1+\eigen_j^i(W)) -
  \frac{\eigen_j^i(W)}{1+\eigen_j^i(W)}\right) + \|\bmumap^i(W) - \bmupr \|^2_{\bCpr^{-1}} \right].
 \end{equation} 

\subsection{Fixed MAP point approximation}
\label{sec:fix-map}
To evaluate the Laplace approximation of the EIG in \eqref{eq:LA-EIG} for each design matrix $W$, one has to solve the optimization problem \eqref{eq:Laplace-mean} for the MAP point $\map^i(W)$ at each data $\bs{y}_i$, $i = 1, \dots, N_s$, which is typically very expensive. 
Rather than solving the optimization problem for each design $W$ and each data $\bs{y}_i$, we consider the MAP point $\bmap^i(\wall)$ at data $\bs{y}_i$ observed from all candidate sensors $\wall$, which is fixed with respect to the design $W$. Consequently, the eigenvalues $\eigen_j^i$ of the Hessian only depend on $W$ through the Jacobian $\bJ$, not on the MAP point. With such a fixed MAP approximation, we can define the approximate EIG as 
\begin{equation}
\tilde{ \Psi } : = \frac{1}{N_s}\sum_{i=1}^{N_s}  \frac{1}{2}
\left[\sum_{j=1}^{k_i}\left( \ln(1+\eigen_j^i(W)) -
\frac{\eigen_j^i(W)}{1+\eigen_j^i(W)}\right) + \|\bmumap^i(\wall) - \bmupr \|^2_{\bCpr^{-1}} \right].
\end{equation}
Since $ \|\bmumap^i(\wall) - \bmupr \|^2_{\bCpr^{-1}}$ is independent of $W$,  maximizing $\tilde{ \Psi } $ is equivalent to maximizing 
\begin{equation}
 \tilde{\psi}  :=  \frac{1}{N_s}\sum_{i=1}^{N_s}  \frac{1}{2}
\left[\sum_{j=1}^{k_i}\left( \ln(1+\eigen_j^i(W)) -
 \frac{\eigen_j^i(W)}{1+\eigen_j^i(W)}\right) \right].
  \end{equation}


\subsection{Prior sample point approximation}
\label{sec:prior-sample}
With the fixed MAP point approximation, one still needs to solve the optimization problem \eqref{eq:Laplace-mean} to generate the MAP points $\bmap^i(\wall)$ from the observation data $\bs{y}_i$, $i = 1, \dots, N_s$. Note that the observation data $\bs{y}_i$ in \eqref{eq:obsi} is provided by first solving the forward model at a sample $\bs{m}_i$ randomly drawn from the prior, and then extracting the observation from all the candidate sensors. As the number of candidate sensors becomes big, we expect that the fixed MAP point $\bmap^i(\wall)$ can recover or come close to the prior sample $\bs{m}_i$ if the inverse problem is not severally ill-posed, which provides a rational to replace the fixed MAP point  $\bmap^i(\wall)$ by the prior sample $\bs{m}_i$. By this prior sample point approximation, one can completely get rid of solving the optimization problem \eqref{eq:Laplace-mean}.


\subsection{Greedy algorithms}
To this end, we present the standard and swapping greedy algorithms for nonlinear Bayesian inverse problems with all the above approximations taken into account. 
  
\begin{algorithm}[H]
\small
   \caption{A standard greedy algorithm }
   \label{alg:greedy12}
\begin{algorithmic}[1]
  \STATE {\bfseries Input}: data $\{ \obs_i\}^{N_s}_{i=1}$ generated from the prior samples $ \{ \bs{m}_i\}^{N_s}_{i=1}$, $d$ candidates set $S$, $r$ budget, initial set $S^0 = \emptyset$.
   \FOR{$i = 1,\dots,N_s$}
   \STATE Form $\mathbf{H}_d = \hat{\bJ}_d(\bs{m}_i) \bCpr \hat{\bJ}^T_d(\bs{m}_i) $.
   \STATE Compute $\mathbf{H}_d \approx U_k^i \Sigma_k^i (U_k^{i})^T$ by Algorithm \ref{alg:RSVD}.
   \ENDFOR
   \FOR{t = 1,\dots,r}
   \STATE $v^* \leftarrow \arg \max_{v \in S \setminus S^{t-1}}\tilde{\psi} (y,W, \{ U^i_k, \Sigma_k^i\}_{i=1}^{N_s})$,\\
   $W_v$ is the design matrix for the sensor choice $S^{t-1} \cup \{ v\}$. 
   \STATE $S^t \leftarrow S^{t-1} \cup \{ v^*\}$.
  \ENDFOR
  \STATE {\bfseries Output:} optimal sensor choice $S^r$.
    \end{algorithmic}
\end{algorithm}
  \begin{algorithm}[H]
\small
   \caption{A swapping greedy algorithm }
   \label{alg:greedy22}
\begin{algorithmic}[1]
  \STATE {\bfseries Input}: data $\{\obs_i\}^{N_s}_{i=1}$ generated from the prior samples $ \{ \bs{m}_i\}^{N_s}_{i=1}$, $d$ candidates set $S$, $r$ budget.
   \FOR{$i = 1,\dots,N_s$}
   \STATE Form $\mathbf{H}_d = \hat{\bJ}_d(\bs{m}_i) \bCpr \hat{\bJ}^T_d(\bs{m}_i) $.
	\STATE Compute $\mathbf{H}_d \approx U_k^i \Sigma_k^i (U_k^{i})^T$ by Algorithm \ref{alg:RSVD}.
   \ENDFOR
  \STATE Compute the initial guess sensor choice $S^*=\{s_1,\dots,s_r\} \subset S$ based on leverage scores of $\sum^{N_s}_{i=1} U_k^i $. $S^0 = \{\emptyset \}$.
  \WHILE{$S^* \neq S^0$}
  \STATE $S^0 \leftarrow S^*$
   \FOR{t = 1,\dots,r}
   \STATE $v^* \leftarrow \arg \max_{v \in \{ s_t\} \cup (S \setminus S^{t-1})}\tilde{\psi} (y,W, \{ U^i_k, \Sigma_k^i\}_{i=1}^{N_s})$,\\
   $W_v$ is the design matrix for the sensor choice $S^{t-1}  \setminus  \{ s_t\}\cup \{ v\}$.
   \STATE $S^t \leftarrow (S^{t-1}  \setminus  \{ s_t\}) \cup \{ v^*\}$.\\
  \ENDFOR\\
  $S^* \leftarrow S^r$.
  \ENDWHILE
  \STATE {\bfseries Output:} optimal sensor choice $S^*$.
    \end{algorithmic}
\end{algorithm}

\subsection{Computation and complexity}
\label{sec:complexity}
In this section, we present the computation involved in finding the MAP point in \eqref{eq:Laplace-mean} and in the low rank approximation \eqref{eq:Hd-low-rank}, which requires PDE solves and takes the dominant computational cost. We also present a comparison of computational complexity for different approximations introduced above. 

\subsubsection{Finding the MAP point}
In the computation of the MAP point, one needs to solve an optimization problem. We use an inexact Newton-CG method, which requires the computation of the action of the Hessian  of the objective \eqref{eq:Laplace-mean} in a given direction $\hat{{m}}$, evaluated at a point ${m}$, which can be formally written as 
\begin{equation}\label{eq:Hhatm}
\H(m) \hat{m} = (\nabla_m^2 \Phi(m, \obs) + \Cpr^{-1}) \hat{m}.
\end{equation}
The second term can be evaluated by \eqref{eq:bCpr} as 
\begin{equation}
\Cpr^{-1} \hat{m} = \bCpr^{-1} \hat{\bs{m}} = \mathbf{A} \mathbf{M}^{-1} \mathbf{A} \hat{\bs{m}},
\end{equation}
which requires solve of a linear system with mass matrix $\mathbf{M}$. The first term of \eqref{eq:Hhatm} can be evaluated by a Lagrangian multiplier method. The potential $ \Phi(m, \obs)$ can be explicitly written as $\Phi(u) = \frac{1}{2}||\mathcal{B}(u) - \obs||^2_{\Gnoise^{-1}}$ by slight abuse of notation. We introduce the Lagrangian 
\begin{equation}
\cL(u, m, v) = \Phi(u) + r(u, m, v),
\end{equation}
where $v$ is the Lagrangian multiplier, $r(u, m, v)$ is the weak form of the forward PDE \eqref{eq:forward-weak}. By taking variation of $\cL$ w.r.t.\ $v$ as zero, we obtain the forward problem \eqref{eq:forward-weak}, which can be equivalently written as: find $u \in \mathcal{V}$ such that  
\begin{equation}\label{eq:L-u}
\langle \tilde{v}, \partial_v r (u, m, v) \rangle = 0 \quad \forall \tilde{v} \in \mathcal{V}.
\end{equation}
By taking variation of $\cL$ w.r.t.\ $u$ as zero, we obtain the adjoint problem: find $v \in \mathcal{V}$ such that 
\begin{equation}\label{eq:L-v}
\langle \tilde{u}, \partial_u r (u, m, v) \rangle = -\langle \tilde{u}, \partial_u \Phi(u) \rangle \quad \forall \tilde{u} \in \mathcal{V}.
\end{equation}
where $\langle \tilde{u}, \partial_u \Phi(u) \rangle = (\mathcal{B}(u) - \obs)^T \Gnoise^{-1} \mathcal{B}(\tilde{u})$. The the gradient of $\Phi$ w.r.t.\ $m$ can be evaluated as 
\begin{equation}\label{eq:L-m}
\langle \tilde{m}, \nabla_m \Phi(m, \obs) \rangle = \langle \tilde{m}, \partial_m \cL(u, m, v) \rangle = \langle \tilde{m}, \partial_m r(u, m, v)\rangle.
\end{equation}
To compute the Hessian action $\nabla_m^2 \Phi \; \hat{m}$, we introduce another Lagrangian by taking the constrains \eqref{eq:L-u}, \eqref{eq:L-v}, and \eqref{eq:L-m} into account as 
\begin{equation}
\cL^H(u, m, v, \hat{u}, \hat{m}, \hat{v}) = \langle \hat{m}, \partial_m r(u, m, v) \rangle + \langle \hat{v}, \partial_v r(u, m, v) \rangle + \langle \hat{u}, \partial_u r(u, m, v) +\partial_u \Phi(u) \rangle,
\end{equation}
where $\hat{u}$ and $\hat{v}$ are the Lagrangian multipliers.  By taking variation of $\cL^H$ w.r.t.\ $v$ as zero, we obtain the incremental forward problem: find $\hat{u} \in \mathcal{V}$ such that 
\begin{equation}\label{eq:Lhatu}
\langle \tilde{v}, \partial_{uv} r \; \hat{u}\rangle = -\langle \tilde{v}, \partial_{mv} r \; \hat{m} \rangle \quad \forall \tilde{v} \in \mathcal{V},
\end{equation}
where $\partial_{uv} r: \mathcal{V} \to \mathcal{V}'$ and $\partial_{mv} r: \mathcal{M} \to \mathcal{V}'$ are the linear operators.  By taking variation of $\cL^H$ w.r.t.\ $u$ as zero, we obtain the incremental adjoint problem: find $\hat{v} \in \mathcal{V}$ such that 
\begin{equation}\label{eq:Lhatv}
\langle \tilde{u}, \partial_{vu} r \; \hat{v}\rangle = -\langle \tilde{u}, \partial_{mu} r \; \hat{m} \rangle   -\langle \tilde{u}, \partial_{uu} r \; \hat{u} \rangle -\langle \tilde{u}, \partial_{uu} \Phi \; \hat{u} \rangle \quad \forall \tilde{u} \in \mathcal{V},
\end{equation}
where $\partial_{vu} r: \mathcal{V} \to \mathcal{V}'$, $\partial_{mu} r: \mathcal{M} \to \mathcal{V}'$, and $\partial_{uu} r: \mathcal{V} \to \mathcal{V}'$ are linear operators, and $\langle \tilde{u}, \partial_{uu} \Phi \; \hat{u} \rangle  =  \mathcal{B}(\hat{u})^T \Gnoise^{-1} \mathcal{B}(\tilde{u})$. To this end, the Hessian action $\nabla_m^2 \Phi \; \hat{m} $ can be evaluated as 
\begin{equation}\label{eq:Lhatm}
\langle \tilde{m}, \nabla_m^2 \Phi \; \hat{m} \rangle = \langle \tilde{m}, \partial_m \cL^H \; \hat{m} \rangle = \langle \tilde{m}, \partial_{mm} r \; \hat{m} + \partial_{vm} r \; \hat{v} + \partial_{um} r\; \hat{u} \rangle, 
\end{equation}
where $\partial_{mm} r: \mathcal{M} \to \mathcal{M}'$, $\partial_{vm} r: \mathcal{V} \to \mathcal{M}'$, and $\partial_{um} r: \mathcal{V} \to \mathcal{M}'$ are linear operators. Therefore, at $m$, after solving  one forward problem \eqref{eq:L-v} and one adjoint problem \eqref{eq:L-u}, the Hessian action $\nabla_m^2 \Phi \; \hat{m}$ at each $\hat{m}$ requires solution of one incremental forward problem \eqref{eq:Lhatu} and one incremental adjoint problem \eqref{eq:Lhatv}, i.e., two linearized PDE solves. We solve all the above PDEs by a finite element method in the subspace $\mathcal{V}_{n_u} \subset \mathcal{V}$, and $\mathcal{M}_n \subset \mathcal{M}$.

\subsubsection{Low rank approximation}
In the computation of the low rank approximation \eqref{eq:Hd-low-rank} by Algorithm \ref{alg:RSVD}, we need to perform the actions $\mathbf{H}_d \bs{\Omega}$ and $\mathbf{H}_d \bs{Q}$, where $\mathbf{H}_d$ is defined in \eqref{eq:Hd-low-rank}. We present the computation of the action $\mathbf{H}_d \hat{\bs{z}}$ in a direction $\hat{ \bs{z}} \in \bR^d$.
By definition, we have 
\begin{equation}\label{eq:Hd}
\mathbf{H}_d \hat{\bs{z}} = \hat{\bJ}_d \bCpr \hat{\bJ}^T_d \hat{\bs{z}} =  (\Gnoise^d)^{-\frac{1}{2}}  \bJ_d \bCpr \bJ_d ^T (\Gnoise^d)^{-\frac{1}{2}} \hat{\bs{z}}  =  (\Gnoise^d)^{-\frac{1}{2}}  \bJ_d \mathbf{A}^{-1} \mathbf{M} \mathbf{A}^{-1}\bJ_d^T  (\Gnoise^d)^{-\frac{1}{2}} \hat{\bs{z}}, 
\end{equation}
which involves the computation of the actions of the Jacobian ${\bJ}_d$ and its transpose ${\bJ}_d^T$, and two solves of a linear system with stiffness matrix $\mathbf{A}$. We first consider the action of ${\bJ}_d^T \bs{z}$ in direction $\bs{z} = (\Gnoise^d)^{-\frac{1}{2}} \hat{ \bs{z}}$. By definition of the Jacobian in \eqref{eq:Jacobian} and $\mathcal{F}(m) = \mathcal{B}_d(u(m))$, we have  
${\bJ}_d^T \bs{z} = (\nabla_{\bs{m}} \bs{u}(\bs{m}))^T \mathbf{B}_d^T \bs{z}$,
where the $j$-th column $(\mathbf{B}_d)_j = \mathcal{B}_d(\psi_j) \in \bR^d$ for the basis functions $\{\psi_j\}_{j=1}^{n_u}$ in approximating the state $u = \sum_{j = 1}^{n_u} u_j \psi_j$, and $\bs{u} = (u_1, \dots, u_{n_u})^T \in \bR^{n_u}$ is the coefficient vector. 
By taking the gradient of the forward problem in the form \eqref{eq:L-u} w.r.t.\ $m$, and noting that $u$ depends on $m$, we have
\begin{equation}
0 = \langle \tilde{v}, \partial_{vu} r \; \nabla_m u(m) + \partial_{vm} r \rangle \quad \forall v \in \mathcal{V}.
\end{equation}
Let $\mathbf{R}_{vu} $ and $\mathbf{R}_{vm}$ denote the matrices corresponding to the finite element discretization of $\partial_{vu} r$ and $\partial_{vm} r$ above, respectively, we formally obtain
\begin{equation}
 \nabla_{\bs{m}} \bs{u}(\bs{m}) = - \mathbf{R}_{vu}^{-1} \mathbf{R}_{vm}, 
\end{equation}
so that ${\bJ}_d^T \bs{z}$ can be evaluated by solving the linear system $\mathbf{R}_{vu}^{T} \bs{w} = \mathbf{B}_d^T \bs{z}$ for $\bs{w} \in \bR^{n_u}$, which has the same matrix $\mathbf{R}_{vu}^{T} = \mathbf{R}_{uv}$ as in the discrete incremental forward problem \eqref{eq:Lhatu}, and performing the matrix-vector product $-\mathbf{R}_{vm}^T \bs{w}$. Similarly the Jacobian action $\bJ_d \bs{n} = \mathbf{B}_d \nabla_{\bs{m}} \bs{u}(\bs{m}) \bs{n}$ for $\bs{n} = \mathbf{A}^{-1} \mathbf{M} \mathbf{A}^{-1}\bJ_d^T {\bs{z}}$ can be evaluated by first performing the matrix-vector product $\mathbf{R}_{vm} \bs{n}$, then solving the linear system $\mathbf{R}_{vu} \bs{w} = \mathbf{R}_{vm} \bs{n}$ for $\bs{w}  \in \bR^{n_u}$, which has the same matrix $\mathbf{R}_{vu} $ as in the discrete incremental adjoint problem \eqref{eq:Lhatv}, and finally performing the matrix-vector product $-\mathbf{B}_d \bs{w}$. In summary, after solving one forward problem \eqref{eq:L-v} for $u$ and one adjoint problem \eqref{eq:L-u} for $v$, each action $\mathbf{H}_d \hat{\bs{z}}$ consists of solving four linearized PDEs, one with matrix $\mathbf{R}_{vu}$, one with $\mathbf{R}_{uv}$, and two with $\mathbf{A}$.


\subsubsection{Computational complexity}

Suppose we need to evaluate the approximate KL divergence \eqref{eq:Phiapprox} at each of the $N_s$ training data $\{\obs_i\}_{i = 1}^{N_s}$ for $N_{kl}$ times, each time corresponding to a different choice of the sensor location to find the optimal design by the greedy algorithms. Assume that to find the MAP point for each training data, we need $N_{nt}$ Newton iterations and $N_{cg}$ CG iterations for each Newton iteration in average.  Then in total we need $N_{kl} \times N_s \times  N_{nt} \times N_{cg}$ times Hessian action \eqref{eq:Hhatm} to compute the exact MAP points, $N_s \times N_{nt} \times N_{cg}$ times to compute the fixed MAP points, $0$ times if the prior points are used. 
Assume that in average the rank $k$ with oversampling factor $p = 10$ is used in the low rank approximation by Algorithm \ref{alg:RSVD}. Then in total we need $N_{kl} \times N_s \times 2(k+p)$ times Hessian action \eqref{eq:Hd} for the case of exact MAP point, $N_s \times 2(k+p)$ times for the case of fixed MAP point, and $N_s \times 2(k+p)$ times for case of prior point. We summarize the computational complexity in terms of Hessian action in the following table.


\begin{table}[!htbp]
\caption{Computational complexity in terms of the number of Hessian actions \eqref{eq:Hhatm} for finding MAP point and \eqref{eq:Hd} for low rank approximation. $N_s$: \# data, $N_{kl}$: \# EIG evaluations, $N_{nt}$: \# Newton iterations, $N_{cg}$: \# CG iterations, $k$: \# rank of the low rank approximation \eqref{eq:Hd-low-rank},  $p$: an oversampling factor in Algorithm \ref{alg:RSVD}.}
\begin{center}
\begin{tabular}{|c|c|c|c|}
\hline
\# Hessian action  &exact MAP point &fixed MAP point & prior sample point\\
\hline
\tabincell{c}{Finding MAP point\\(by Newton-CG) } & $N_{kl} \times N_s \times  N_{nt} \times N_{cg}$ & $N_s \times  N_{nt} \times N_{cg}$ & $0$ \\
\hline
\tabincell{c}{Low-rank approximation \\ (by RSVD Algorithm \ref{alg:RSVD})}& $N_{kl} \times N_s \times 2(k+p)$ & $N_s \times 2(k+p)$ & $N_s \times 2(k+p)$\\
\hline
\end{tabular}
\end{center}
\end{table}

\section{Numerical results}
\label{sec:numerical}
In this section, we present OED for both a linear and a nonlinear Bayesian inverse problems to demonstrate the efficacy and efficiency of the proposed approximations and the greedy algorithm, as well as the scalability of our method w.r.t.\ the number of training data points, the number of candidate sensors, and the parameter dimension.

 \subsection{A linear Bayesian inverse problem}
 In this example, we consider inversion of the initial condition of an advection-diffusion problem given pointwise observation of the state at certain sensor locations and certain times. The forward problem is given by 
\begin{equation}\label{eq:advection-diffusion}
\begin{split}
u_t - k \Delta u + \bs{v} \cdot \nabla u &= 0 \text{ in } \mathcal{D} \times (0,T),\\
u(\cdot, 0) & = m \text{ in } \mathcal{D} ,\\
k \nabla u \cdot n & = 0 \text{ on } \partial\mathcal{D} \times (0,T),
\end{split}
\end{equation}
where $\mathcal{D} \subset \bR^2$ is an open and bounded domain with boundary $\partial \mathcal{D}$ depicted in Figure \ref{fig:velocity}.  $k >0$ is the diffusion coefficient and $T > 0$ is the final time. In our numerical experiment, we choose $k = 0.001$. The velocity field $\bs{v} \in \bR^2$ is obtained by solving the following steady-state Navier--Stokes equation with the side walls driving the flow:
 \begin{equation}\label{eq:Navier-Stokes}
 \begin{split}
-\frac{1}{\text{Re}} \Delta \bs{v} + \nabla q + \bs{v} \cdot \nabla \bs{v} & = 0 \text{ in } \mathcal{D} ,\\
\nabla \cdot \bs{v} & = 0 \text{ in } \mathcal{D} ,\\
\bs{v} & = \bs{g} \text{ on } \partial\mathcal{D},
 \end{split}
\end{equation}
where $q$ represents the pressure field, the Reynolds number $\text{Re} = 50$. The Dirichlet boundary data $\bs{g} \in \bR^2$ is prescribed as $\bs{g} = (0, 1)$ on the left wall of the domain, $\bs{g} = (0, -1)$ on the right wall, and $\bs{g} = (0, 0)$ elsewhere.
\begin{figure}[!hbtp]
\centering
\includegraphics[width=.4\textwidth]{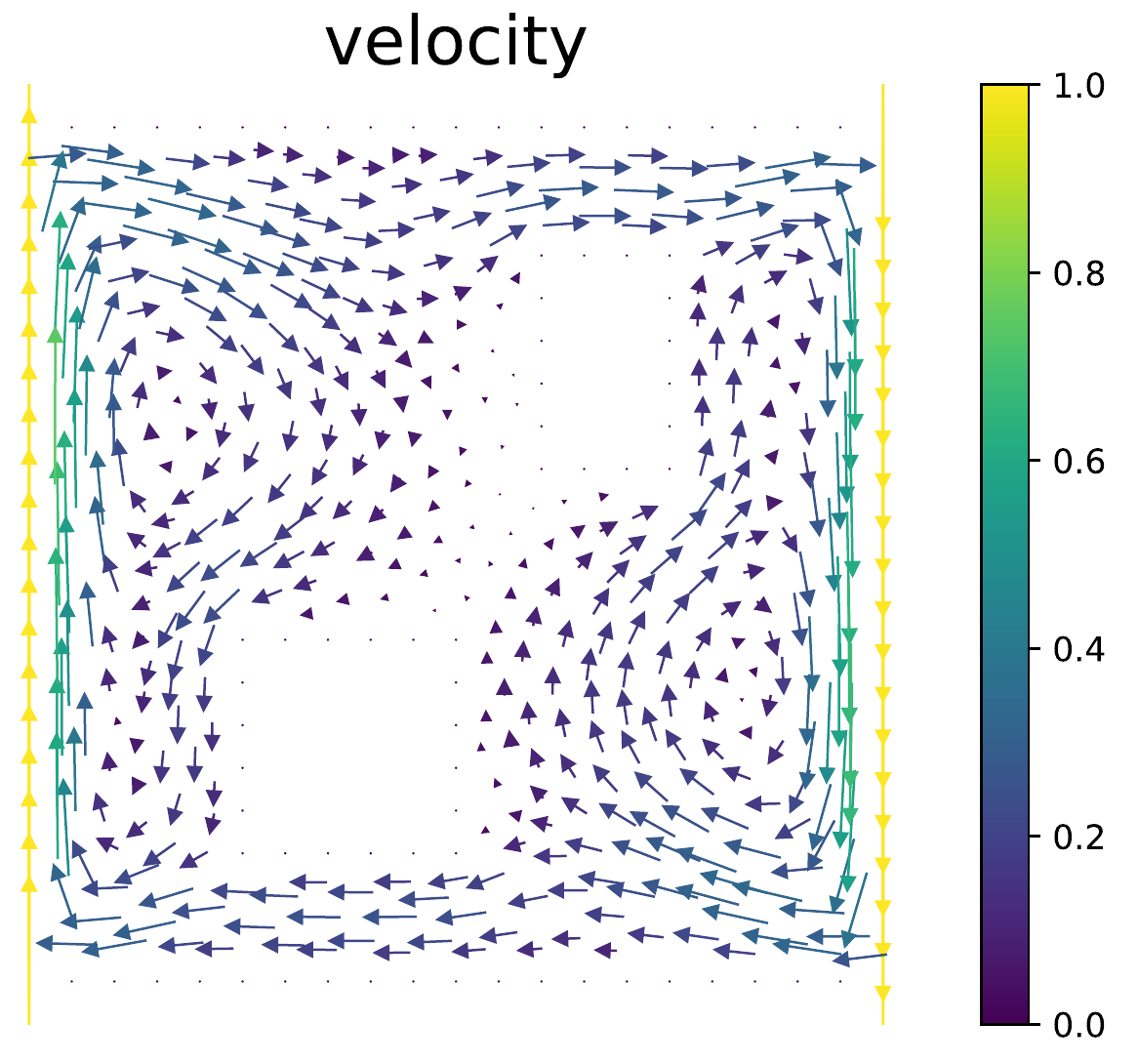} 
\includegraphics[width=.4\textwidth]{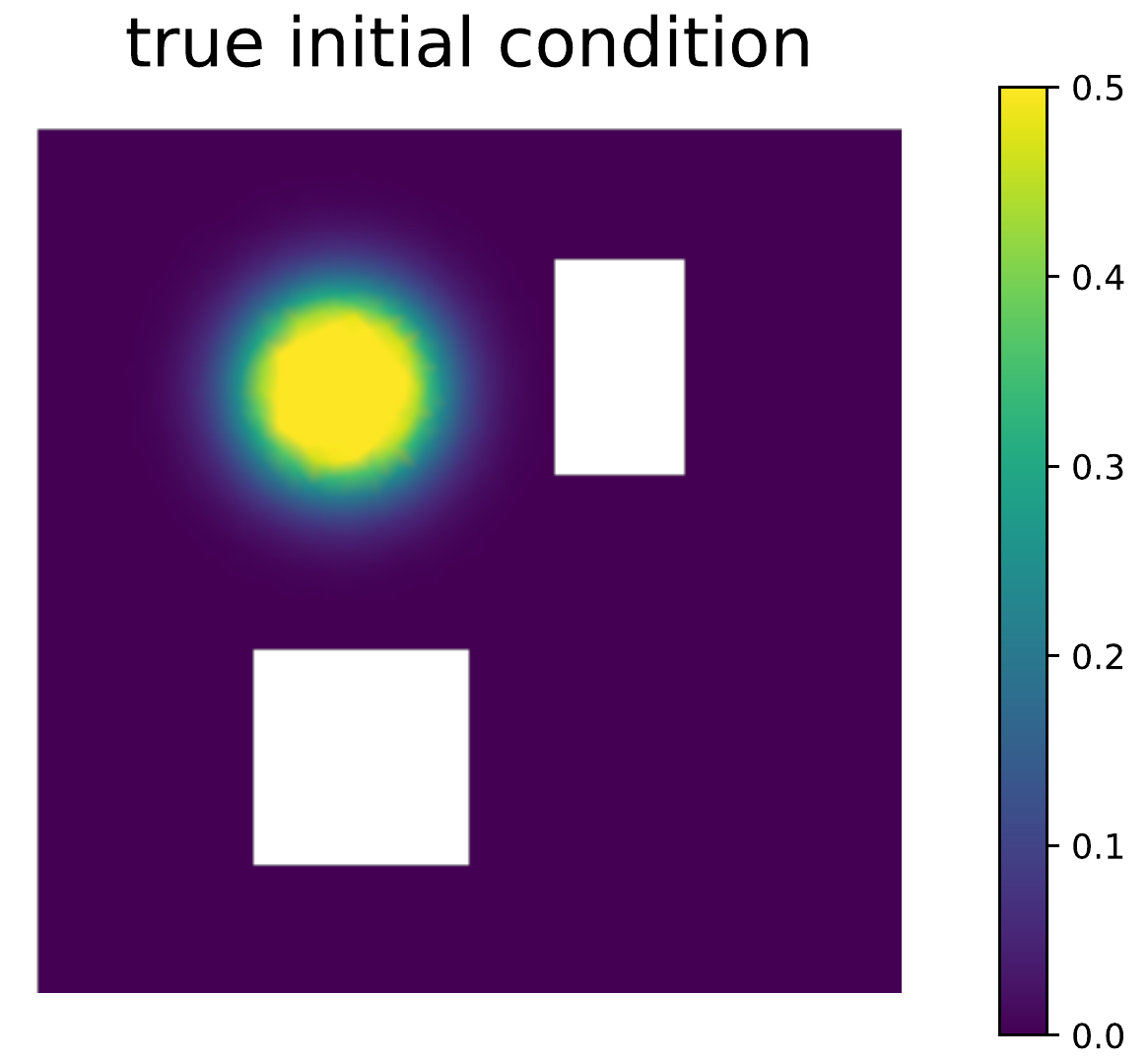} 
\caption{The computational domain $\mathcal{D}$ for two dimensional problem is $[0,1]^2$ with two rectangular blocks ($[0.25,0.5]\times[0.15,0.4], [0.6,0.75]\times[0.6,0.85]$) removed.  Left: velocity field $\bs{v}$. Right:  initial condition $m_{\text{true}}$.}
\label{fig:velocity}
\end{figure}  \\
We use Galerkin finite element method with piecewise linear elements for spatial discretization of the forward and adjoint problems, which results in $n = 2023$ spatial degrees of freedom for the parameter $m$ and state variable $u$. We use implicit Euler method for temporal discretization with $N_t = 40$ time steps for the final time $T = 4$.
We consider a Gaussian prior $ \mupr = \mathcal{N}(\prmean, \Cpr)$ for the parameter $m$. The covariance $\Cpr = \mathcal{A}^{-2}$ is given by the square of the inverse of differential operator $\mathcal{A} = -\gamma \Delta + \delta I$ with Laplacian $\Delta$ and identity $I$, equipped with Robin boundary condition $\gamma \nabla m \cdot \mathbf{n} + \beta m$ on $\partial \mathcal{D}$, where $\gamma, \delta > 0$ control the correlation length and variance of the prior distribution. The Robin coefficient $\beta$ is chosen as in \ref{DaonStadler18} to reduce boundary artifacts. For our numerical test, we choose $\prmean = 0.25$, $\gamma = 1, \delta = 8$, and set a ``true" initial condition $m_{\text{true}} = \min(0.5, \exp(-100\norm{x-[0.35,0.7]}^2)$, as shown in Figure \ref{fig:velocity}, to generate the observation data at the final time, as shown in Figure \ref{fig:obs}.


In the first test, we use a small number of candidate sensors and compare the design obtained by the greedy algorithms with the optimal design by brutal force search to show the efficacy of the greedy algorithms. 
\begin{figure}[!hbtp]
\centerline{\includegraphics[width=1.\textwidth]{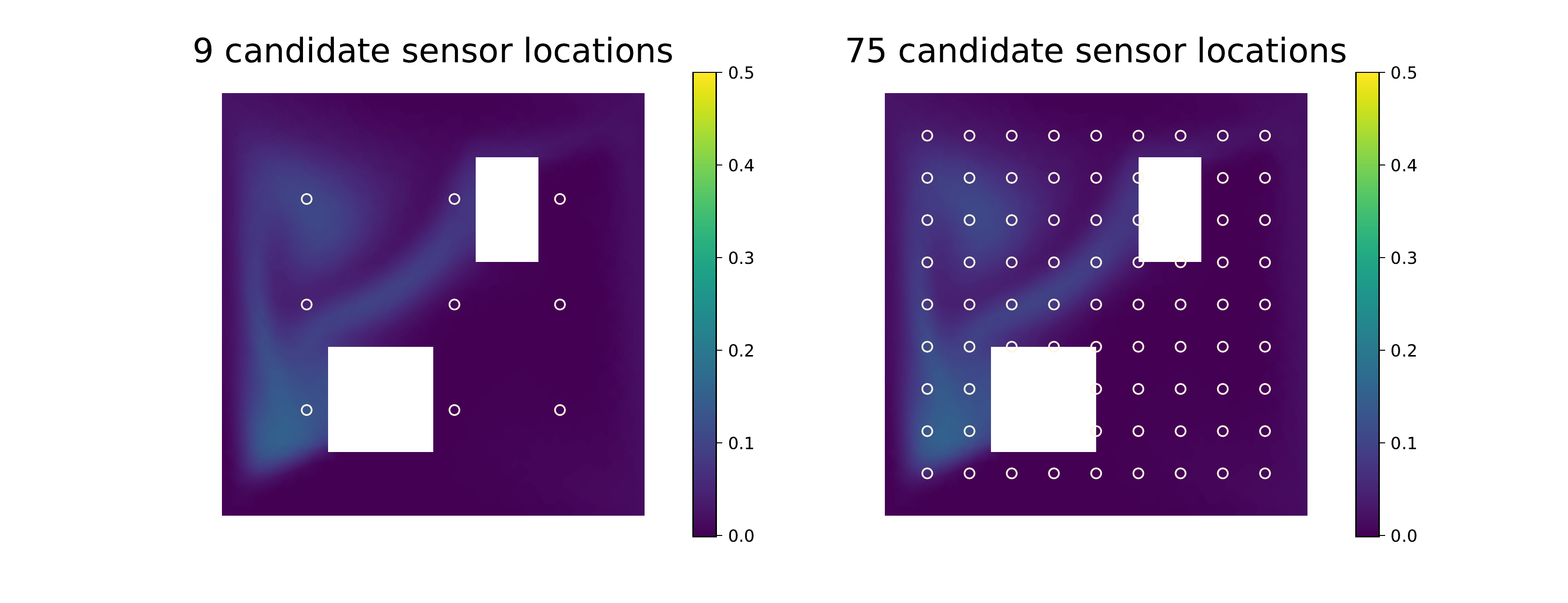}} 
\caption{The state field at time $T = 4$ obtained by solving \eqref{eq:advection-diffusion} at the initial condition in Figure \ref{fig:velocity}. Left: 9 candidate observation locations. Right: 75 candidate observation locations.} 
\label{fig:obs}
\end{figure}   
Specifically, we use a grid of $d = 9$ candidate sensor locations $\{x_i\}_{i=0}^9$ ($x_i \in \{0.2,0.55,0.8 \}\times \{0.25,0.5,0.75 \}$) as shown in Figure \ref{fig:obs} (left) with the goal of choosing $r = 2, 3, 4, 5, 6, 7, 8$ sensors at the finial time. 
We run the two greedy algorithms, Algorithm \ref{alg:greedy1} and Algorithm \ref{alg:greedy2}, as well as a brutal force search of all possible designs ($\frac{9!}{r!(9-r)!}$) to find the optimal design. In the evaluation of the approximate EIG \eqref{eq:Psihat}, we do not need the low-rank approximation of $\mathbf{H}_d$ here as in Theorem \ref{thm:low-rank-EIG} since it is a small ($9 \times 9$) matrix that we can easily compute by solving 9 incremental forward and 9 incremental adjoint problems in directions $e_j$ of dimension 9 with the $j$-th element as one and other as zeros, $j = 1, \dots, 9$.  The values of the approximate EIG by the two algorithms and at all possible designs are shown in Figure \ref{fig:linear-EIG} (left), which are also shown in Table \ref{table:linear} for a closer comparison.

\begin{figure}[h!]
  \centering
  \begin{subfigure}[b]{0.45\linewidth}
    \includegraphics[width=\linewidth]{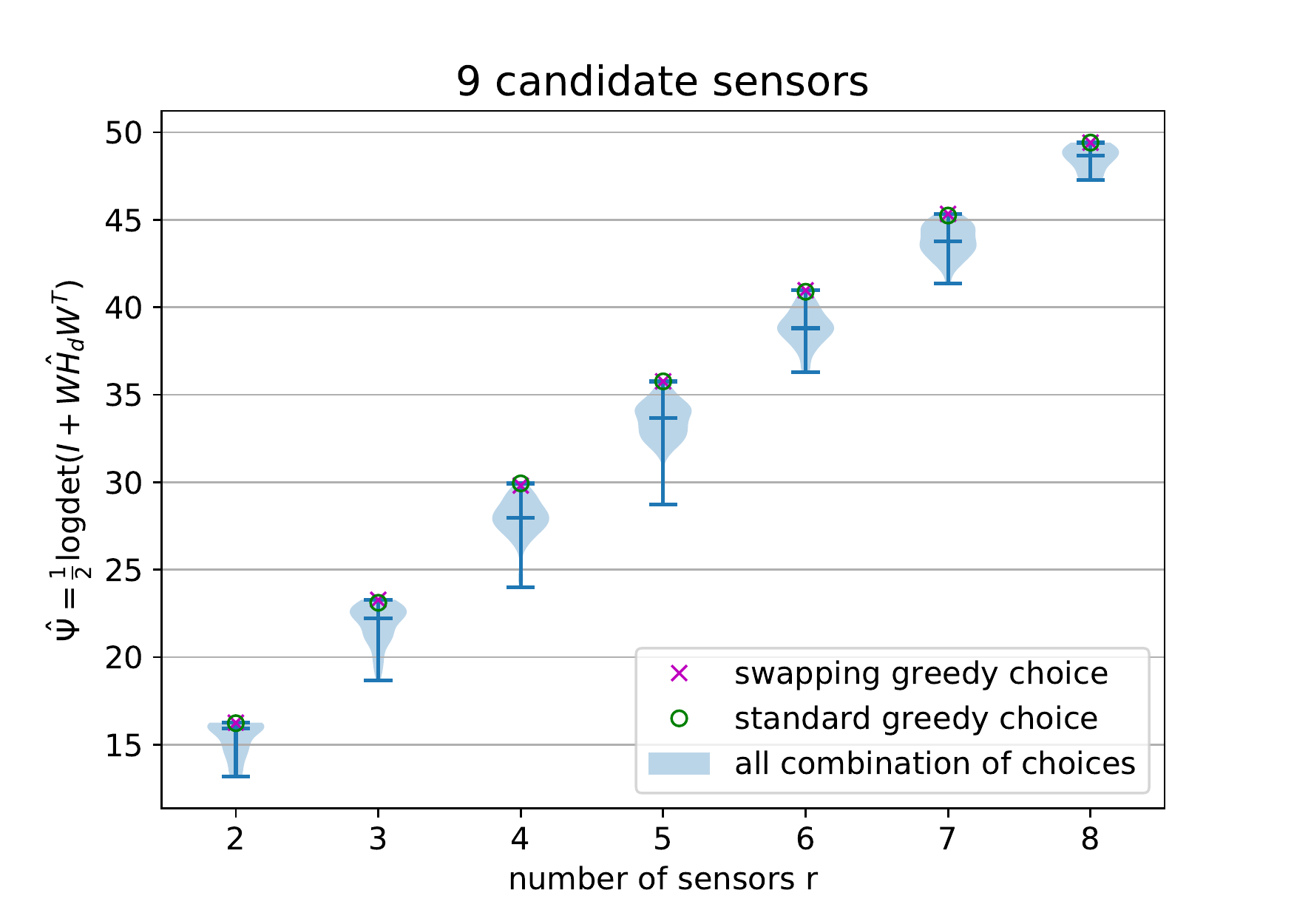}   
  \end{subfigure}
  \begin{subfigure}[b]{0.45\linewidth}
    \includegraphics[width=\linewidth]{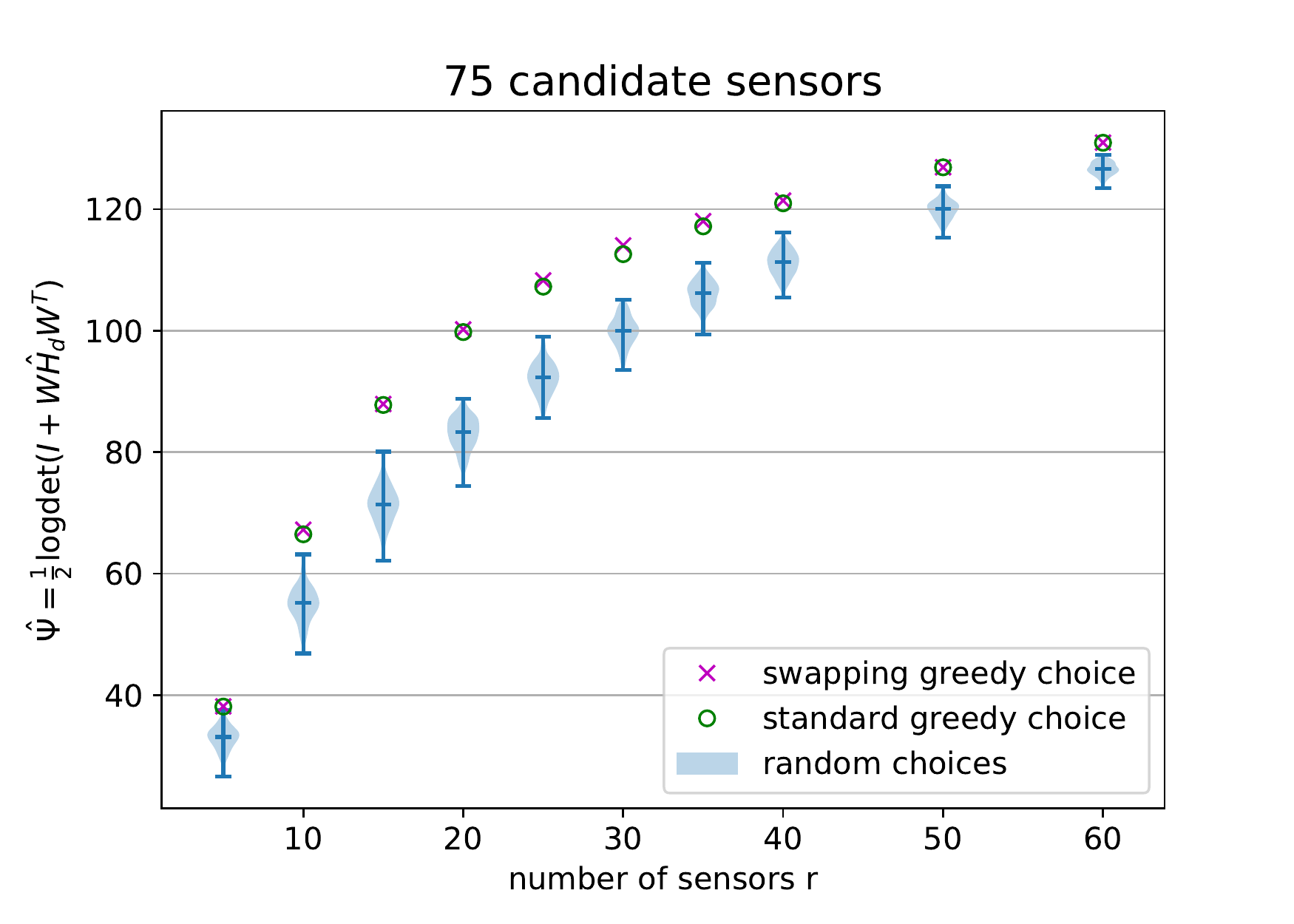}   

  \end{subfigure}
  \caption{Approximate EIG $\hat{\Psi}$ at the design chosen by the standard greedy algorithm, the swapping greedy algorithm and at the optimal design, with the ranking of greedy choices among all the possible designs in the bracket. 
}
\label{fig:linear-EIG}
\end{figure}

\begin{table}[!htbp]
\caption{Approximate EIG $\hat{\Psi}$ at the design chosen by the standard greedy algorithm, the swapping greedy algorithm and at the optimal design, with the ranking of greedy choices among all the possible designs in the bracket. 
}
\begin{center}
\begin{tabular}{|c|c|c|c|c|c|c|c|}
\hline
 $r$ &$2$ &$3$&$4$  & $5$&$6$&$7$ & $8$\\
\hline
\tabincell{c}{standard\\greedy} &$16.2236(\#3)$ &$23.1065(\#5)$&$29.9306(\#1)$  & $35.7619(\#1)$&$40.8833(\#2)$&$45.2398(\#2)$ & $49.4098(\#1)$ \\
\hline
\tabincell{c}{swapping\\greedy} &$16.2528(\#1)$ &$23.2767(\#1)$&$29.8067(\#2)$  & $35.7619(\#1)$&$40.9675(\#1)$&$45.3347(\#1)$ & $49.4098(\#1)$\\
\hline
\tabincell{c}{optimal\\choice} &$16.2528$ &$23.2767$&$29.9306$  & $35.7619$&$40.9675$&$45.3347$ & $49.4098$\\
\hline
\end{tabular}
\end{center}
\label{table:linear}
\end{table}

We can see for $r = 2,3,5,6,7,8$, the swapping greedy algorithm finds the optimal design for all $r$ but $r = 4$ with the second best. While the standard greedy algorithm finds the optimal for $r = 3,4,8$, the second best for $r=6,7$, the third for $r=2$, and the fifth for $r=5$. Although the optimal design are not found in all cases, the the values of the approximate EIG of the chosen ones by the two greedy algorithms are really close to the optimal.

In the second test, we consider a grid of $d = 75$ candidate sensor locations as shown in Figure \ref{fig:obs} (right) with the goal of choosing $r = 5,10,15,20,25,30,35,40,45,50,55,60$ sensor locations.
We compute the low rank approximation \eqref{eq:Hdhat} with the eigenvalues displayed in Figure \ref{fig:lineig}, which decays rapidly, by over five orders of magnitude for the 75 eigenvalues.
\begin{figure}[!hbtp]
\centerline{\includegraphics[width=0.7\textwidth]{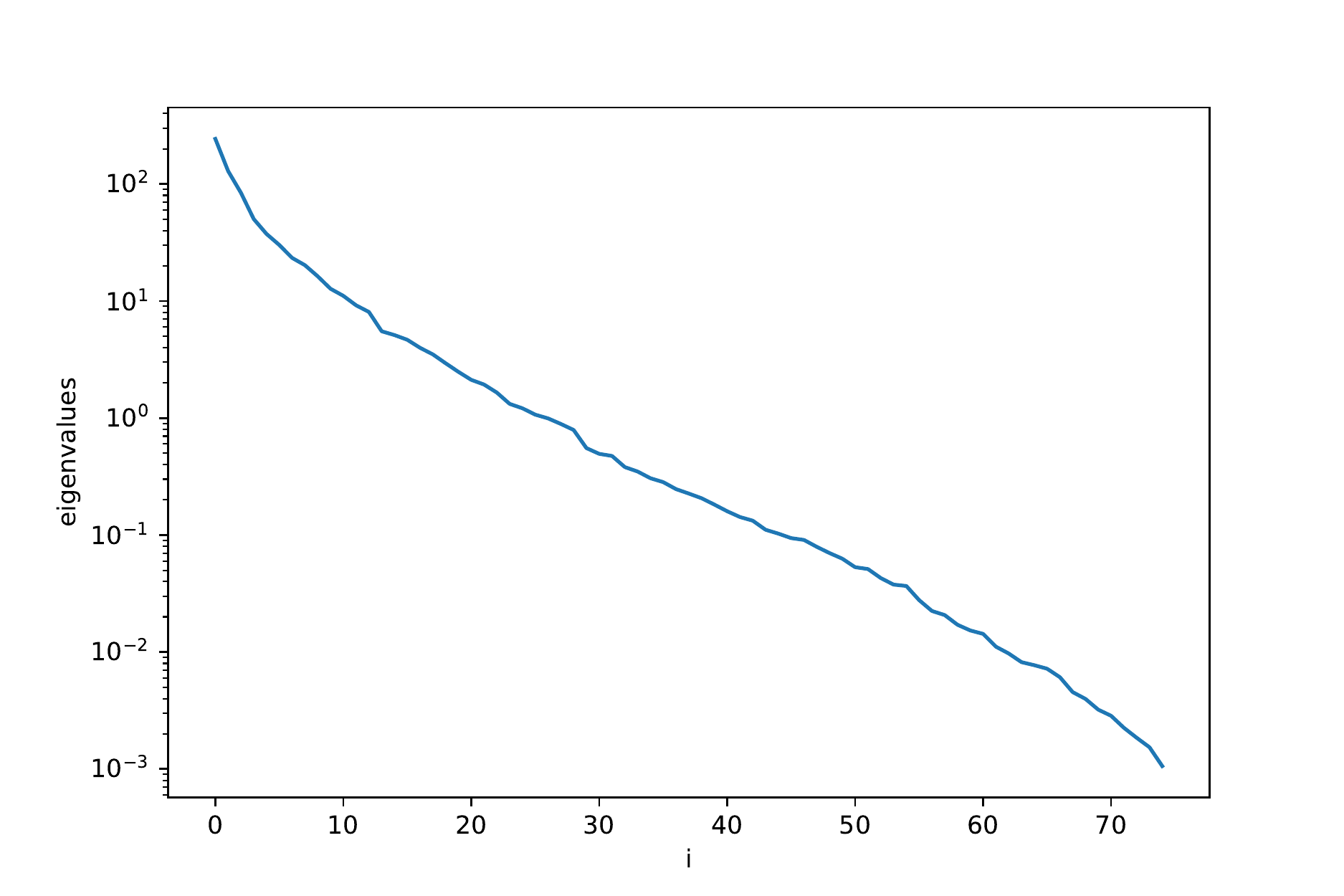}} 
\caption{Decay of the eigenvalues of Hessian $\hat{H}_d$ in \eqref{eq:Hdhat}.} 
\label{fig:lineig}
\end{figure}  
We randomly draw 200 different designs from the candidate sensors and compute their approximate EIG and compare them with the ones chosen by two greedy algorithms as shown in Figure \ref{fig:linear-EIG} (right), from which we can see that both greedy algorithms find the designs better than all the random choices, and the new swapping greedy algorithm we propose always gives higher (better) or at least equal values as shown in Table \ref{table:linear2}.
 \begin{table}[!htbp]
\caption{Approximate EIG $\hat{\Psi}$ at the design chosen by the standard greedy algorithm and the swapping greedy algorithm.}
\begin{center}
\begin{tabular}{|c|c|c|c|c|c|}
\hline
 $r$ &$5$ &$10$&$15$  & $20$&$25$\\
\hline
\tabincell{c}{standard\\greedy} &$38.0898$ &$66.4687$&$87.7688$  & $99.7733$&$107.2547$\\
\hline
\tabincell{c}{swapping\\greedy}&$38.1207$ &$67.2216$&$87.9203$  & $100.2188$&$108.2906$\\
\hline
r&$30$ & $35$& $40$ & $50$ & $60$\\
\hline
\tabincell{c}{standard\\greedy}&$112.6066$ & $117.1837$& $120.9787$ & $126.9114$ & $130.9548$ \\
\hline
\tabincell{c}{swapping\\greedy}&$114.0506$ & $118.0523$& $121.4332$ & $126.9114$ & $130.9821$\\
\hline
\end{tabular}
\end{center}
\label{table:linear2}
\end{table} 
Moreover, the advantage of our swapping greedy algorithm can also be illustrated by reduced pointwise posterior variance in Figure \ref{fig:variance} compared to the standard greedy algorithm, and two random designs with the same number of sensors. 
 \begin{figure}[!hbtp]
\centering
\includegraphics[width=1.\textwidth]{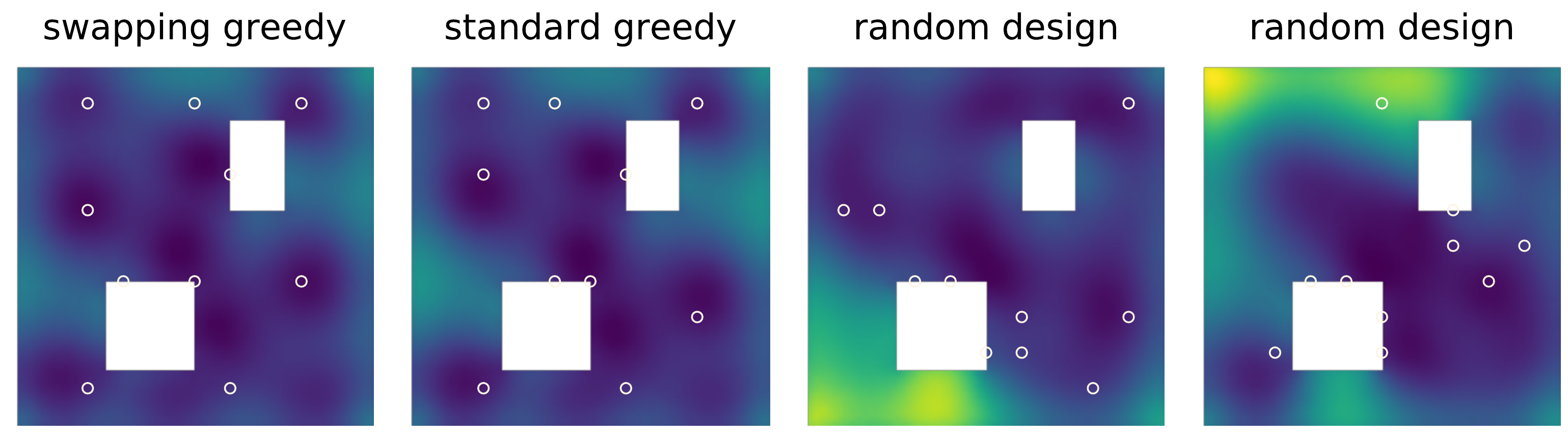}
\caption{Pointwise variance of the posterior at designs chosen by swapping greedy algorithm, the standard greedy algorithm, and two random designs with $10$ sensors. The brighter region corresponds to larger variance. Compared with the optimal design chosen by swapping greedy algorithm, the standard greedy algorithm and the two random designs lead to $7 \%, 30 \%, 53\%$ increase in the averaged variance respectively. } 
\label{fig:variance}
\end{figure}  
 

  \subsection{A nonlinear Bayesian inverse problem}
  \label{sec:nonliner-n}
In this problem we consider a log-normal diffusion forward model as follows
\begin{equation}\label{eq:poison}
\begin{split}
-\nabla \cdot (\exp(m) \nabla u) &= f \text{ in } \mathcal{D}\\
u & = g \text{  on } \Gamma_D \\
\exp(m) \nabla u \cdot \mathbf{n} &= h \text{  on } \Gamma_N
\end{split},
\end{equation}
 where $\mathcal{D} \subset \bR^2$ is an open, bounded domain with sufficiently smooth boundary $\Gamma = \Gamma_D \cup\Gamma_N$ with Dirichlet and Neumann boundaries $ \Gamma_D \cap \Gamma_N = \emptyset$ and data $g \in \mathcal{H}^{1/2}(\Gamma_d)$ and $h \in L^2(\Gamma_N)$ respectively. The state variable $u \in \mathcal{V}_g = \{ v \in \mathcal{H}^1(\mathcal{D}): v|_{\Gamma_D} = g\}$. $f \in L^2(\mathcal{D})$ is a source term. We consider a Gaussian prior for the parameter $m \in \mathcal{H}^1(\mathcal{D})$, i.e.,
$m \sim  \mupr = \mathcal{N}(\prmean, \Cpr)$ with mean $\prmean$ and covariance $\Cpr = \mathcal{A}^{-2}$, where $\mathcal{A}$ is a differential operator given by 
 $$ \mathcal{A} m=
\begin{cases}
-\gamma \nabla \cdot (\Theta \nabla m) + \delta m & \text{ in } \mathcal{D} \\
\Theta \nabla m \cdot \mathbf{n} + \beta m& \text{ on } \partial\mathcal{D}
\end{cases},
$$
where $\beta \approx \sqrt{\gamma \delta}$ is the optimal Robin coefficient derived in \ref{DaonStadler18} to minimize boundary artifacts and $\Theta$ is an symmetric positive definite and anisotropic matrix of the form $$ \Theta = 
\begin{bmatrix}
\theta_1 \sin(\alpha)^2 & (\theta_1 - \theta_2) \sin(\alpha) \cos(\alpha)\\
(\theta_1 - \theta_2) \sin(\alpha) \cos(\alpha) & \theta_2 \cos(\alpha)^2
\end{bmatrix}.$$
In our numerical experiment, we set prior mean as zero, $\gamma = 0.04, \delta = 0.2, \theta_1 = 2, \theta_2 = 0.5, \alpha = \pi / 4$. For the forward problem, we consider the domain $\mathcal{D} = (0,1)\times (0,1)$, no source term (i.e., $f = 0$) and no normal flux on $\Gamma_N = \{0,1 \} \times (0,1)$, i.e., imposing the homogeneous Neumann condition $\exp(m) \nabla u \cdot \mathbf{n} = 0$. The Dirichlet boundary $\Gamma_D = (0,1) \times \{0,1\}$ with boundary condition $u = 1$ on $(0,1) \times \{ 1\}$ and $u = 0$ on $(0,1) \times \{ 0\}$. We draw  a sample from the prior and use it as ``true" parameter field $m_{\text{true}}$ as shown Figure \ref{fig:nonlinear-initial} (left). We use quadratic finite elements for the discretization of the state and adjoint variables, and use linear elements for the parameter. The degrees of freedom for the state and parameter are $n_u = 4225$ and $n = 1089$, respectively.
 \begin{figure}[!hbtp]
\centerline{\includegraphics[width=1.\textwidth]{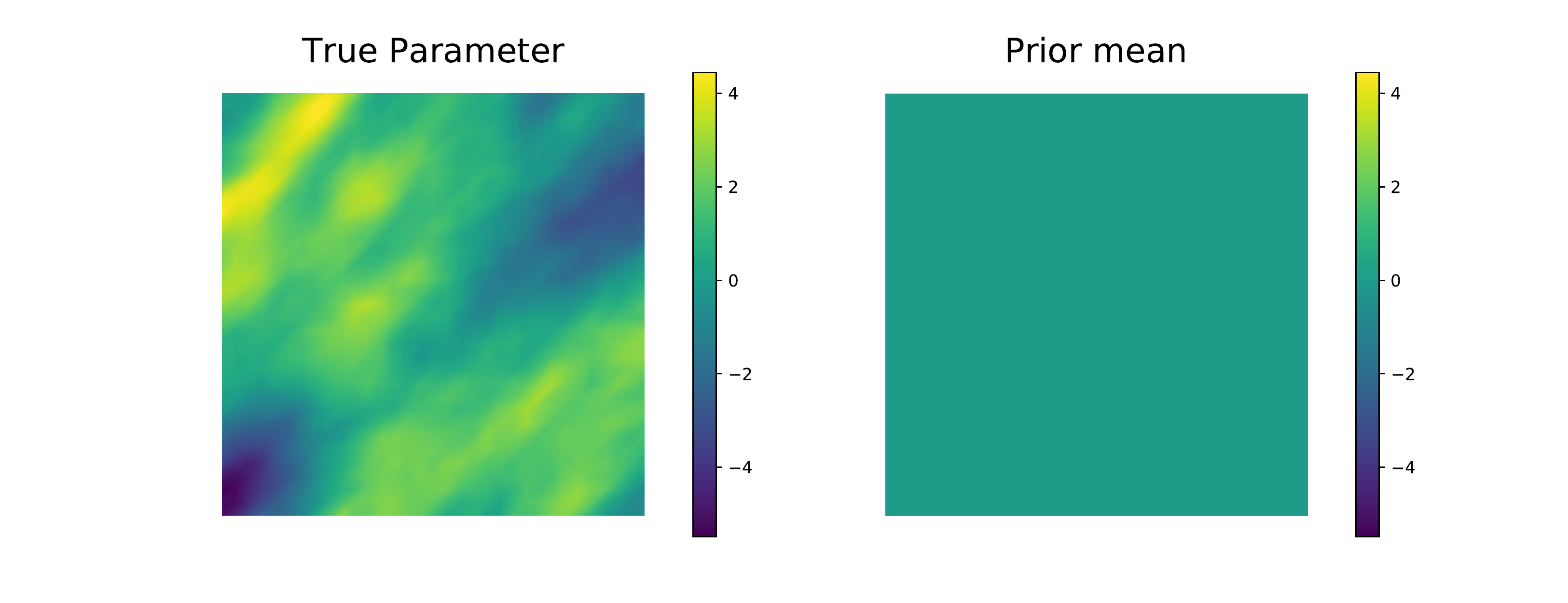}} 
\caption{Left: The synthetic ``true" parameter field. Right: The mean of the prior distribution.} 
\label{fig:nonlinear-initial}
\end{figure}  


 \subsubsection{Effectiveness of approximations} 
 
To reduce the computational cost for the nonlinear inverse problems, we introduced the Laplace approximation with low-rank decomposition, fixed MAP point approximation, and prior sample point approximation in Section \ref{sec:nonlinear}. To investigate their effectiveness, we consider their (sample) correlation at the same design. A high correlation (with correlation coefficient close to $1$) of the approximate EIG values by two different approximations implies that the optimal design obtained by one approximation is likely close to optimal for the other approximation.
\begin{figure}[!hbtp]
\centerline{\includegraphics[width=1.\textwidth]{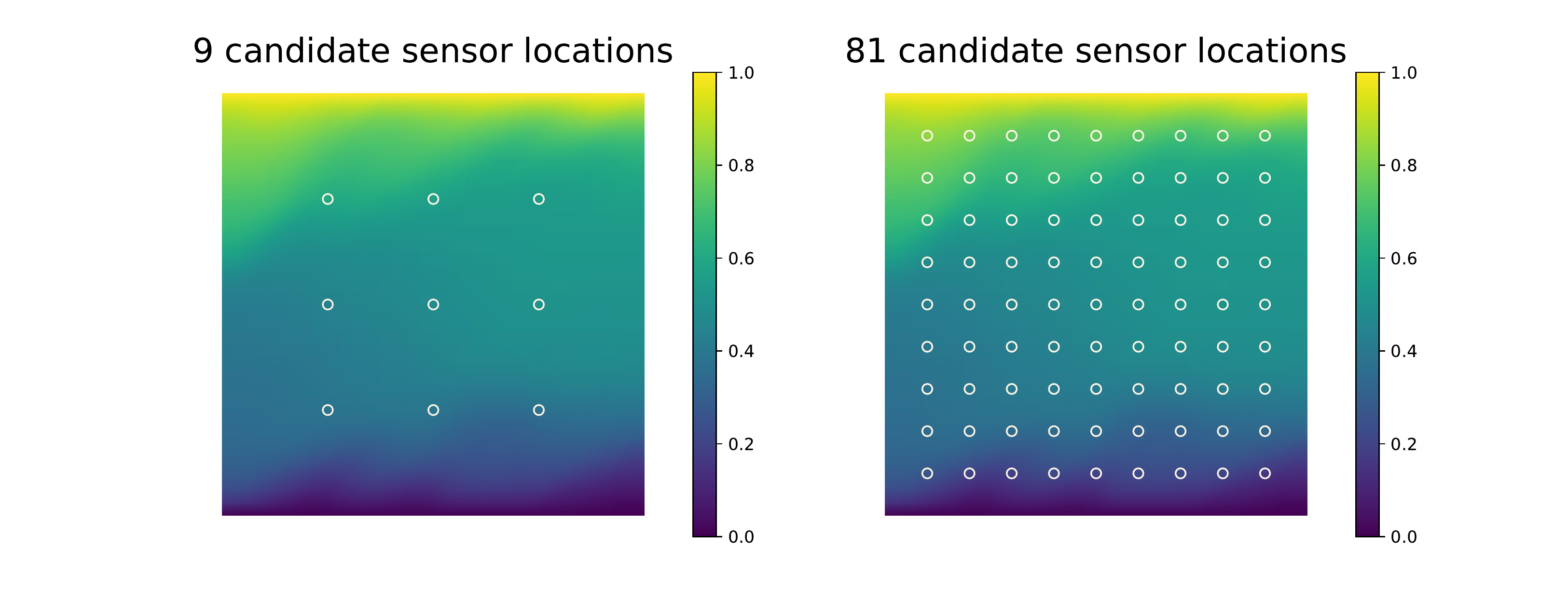}} 
\caption{The state field at the ``true" parameter and two set of candidate observation locations.} 
\label{fig:1}
\end{figure}  

In the test, we use the grid of $d = 81$ candidate sensor locations as shown in Figure \ref{fig:1} (right) with the goal of choosing $r = 10$ sensor locations. We generate $200$ random designs, and compute the EIG of each design by a double-loop Monte Carlo (DLMC) method as the reference. Then we compute their approximate EIG by the Laplace approximation (LA-EIG)  and its further approximation with fixed map point shown in Figure \ref{fig:correlation} (left).  We can see that the correlation between DLMC-EIG and LA-EIG is 0.9752, and the maximum of DLMC-EIG is also the maximum of LA-EIG. The correlation between DLMC-EIG and LA-EIG with fixed map point is 0.9453, and the maximum of LA-EIG with fixed map point is the second maximum of DLMC-EIG. Although it is not the maximum, but it gives almost the same EIG value as the maximum. We can observe closely about the relation between LA-EIG and LA-EIG with fixed MAP point in Figure \ref{fig:correlation} (middle) with a correlation 0.9525, and that their optimal choices have almost the same LA-EIG values. Figure \ref{fig:correlation} (right) illustrates the correlation between $\tilde{\psi}$ computed with fixed MAP point and with prior sample at which to evaluate Hessian. We can see close to $1$ correlation, and that the optimal choices by the fixed map point and the prior sample point are the same.

\begin{figure}[h!]
  \centering
  \begin{subfigure}[b]{0.4\linewidth}
    \includegraphics[width=\linewidth]{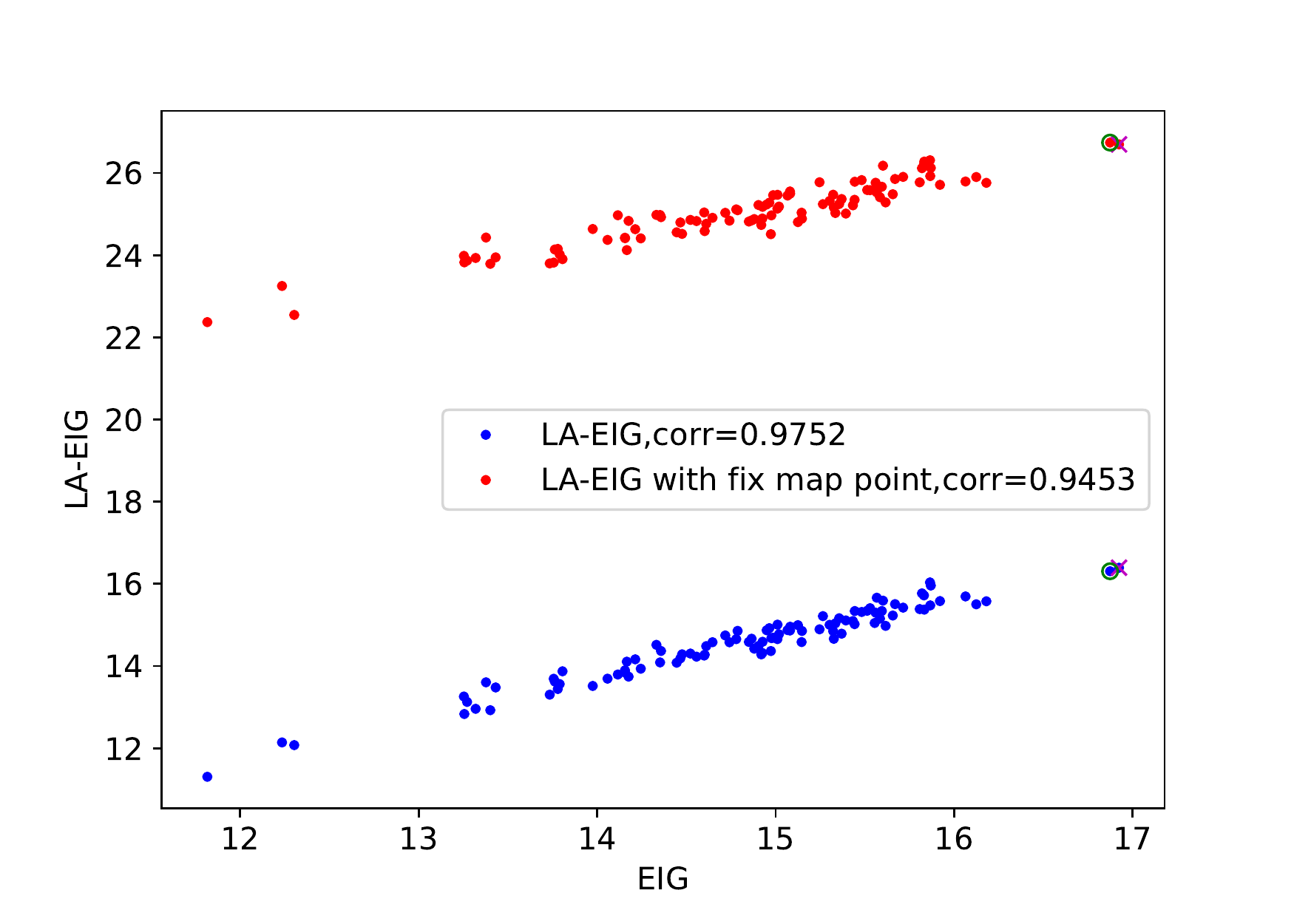}   
  \end{subfigure}
  \begin{subfigure}[b]{0.4\linewidth}
    \includegraphics[width=\linewidth]{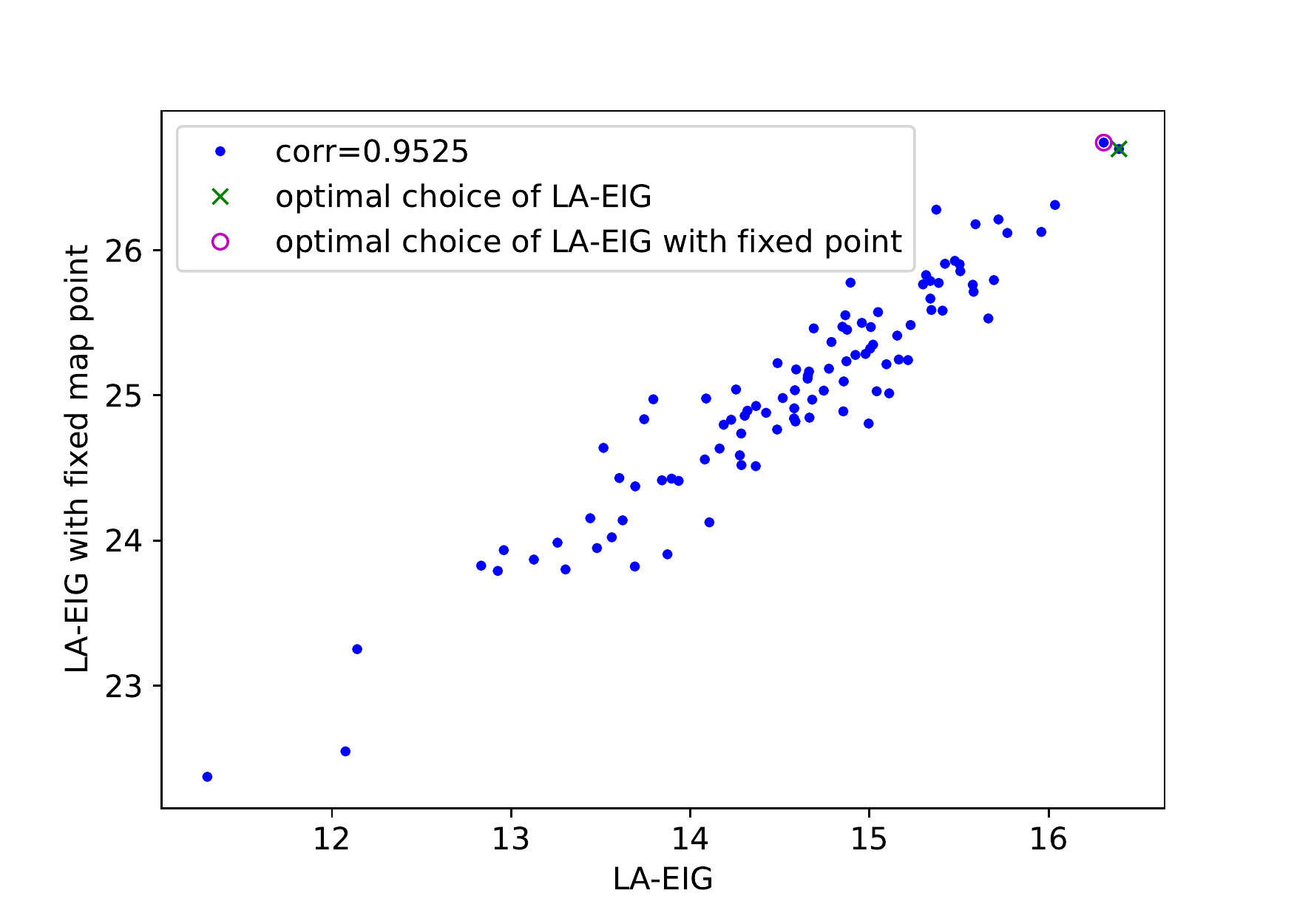}   
  \end{subfigure}
    \begin{subfigure}[b]{0.4\linewidth}
    \includegraphics[width=\linewidth]{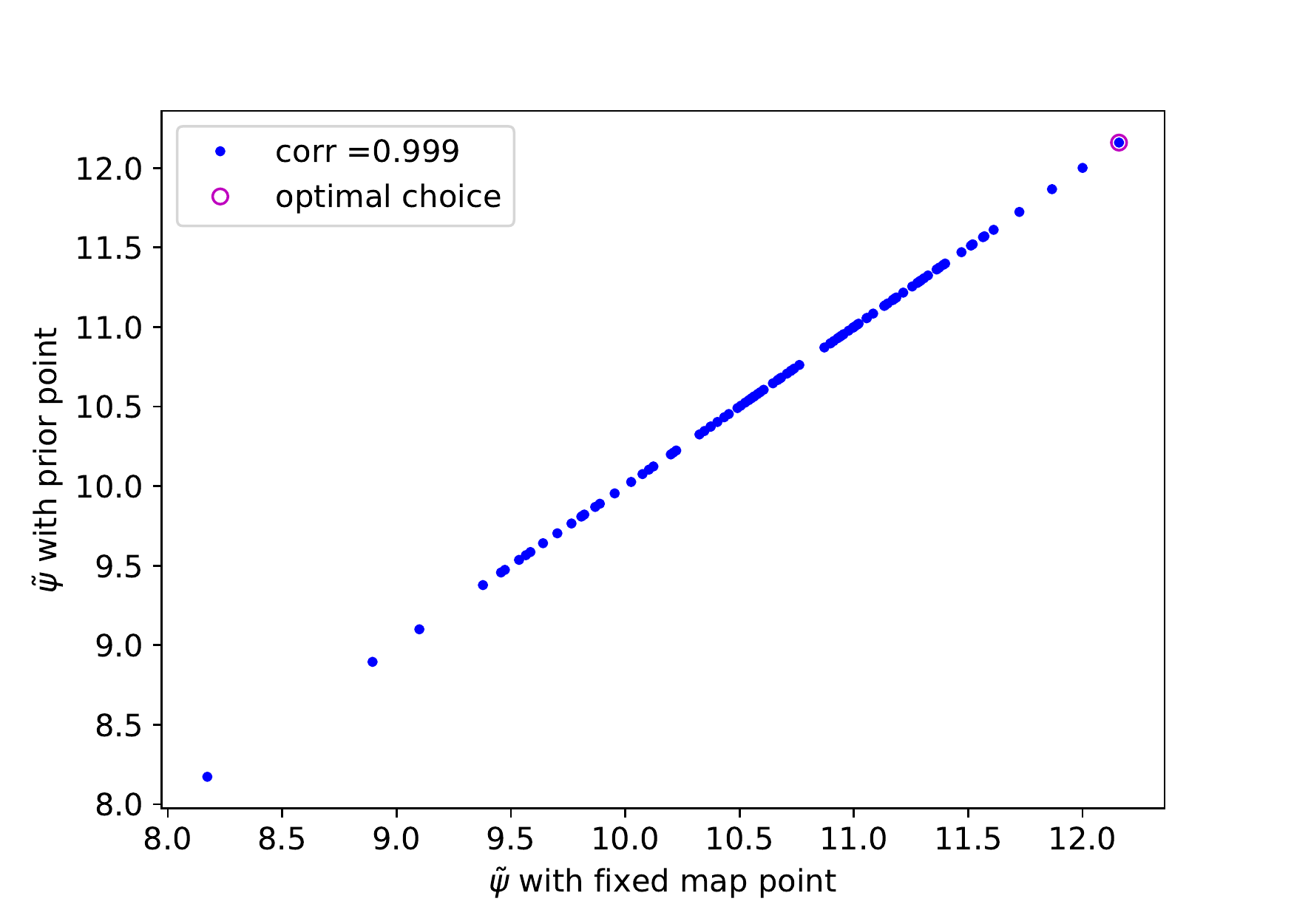}   
  \end{subfigure}
\caption{Correlation between EIG and its approximations. Each blue or red point represents one random design of $10$ sensors among $81$ candidates. Left: DLMC-EIG vs LA-EIG and LA-EIG with fixed MAP point approximation. Middle: LA-EIG and LA-EIG with fixed MAP point approximation. Right: LA-EIG with fixed MAP point approximation vs LA-EIG with prior sample point approximation.
}

\label{fig:correlation}

\end{figure}
 \subsubsection{Numerical results} 
 We first use a grid of $d = 9$ candidate sensor locations with the goal of choosing $r = 2,3,4,5,6,7,8$. As we can see in Figure \ref{fig:poi} (left), the standard greedy and swapping greedy algorithms give the same optimal design for all the cases and they are the actual optimal one among all possible choices for $r = 2,3,5,8$, second best for $r = 4,7$, and third for $r = 6$. We remark that by evaluating the Hessian at the prior sample point, the computational cost is significantly reduced as analyzed in Section \ref{sec:complexity}. Despite that the optimal choice by the approximation might not be the optimal for LA-EIG, we can still find the ones close to the best in all the cases. 

\begin{figure}[h!]
  \centering
  \begin{subfigure}[b]{0.45\linewidth}
    \includegraphics[width=\linewidth]{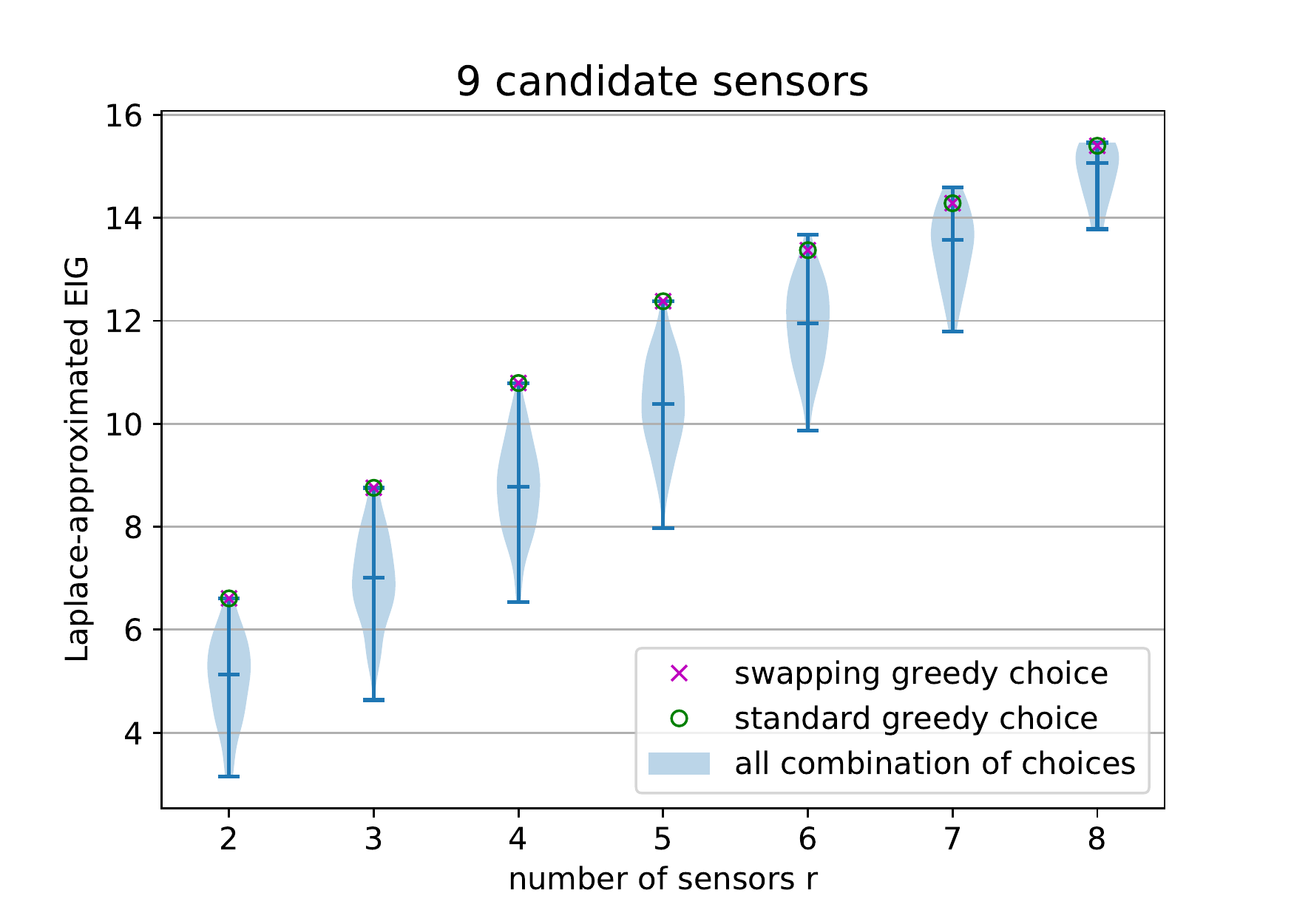}   
  \end{subfigure}
  \begin{subfigure}[b]{0.45\linewidth}
    \includegraphics[width=\linewidth]{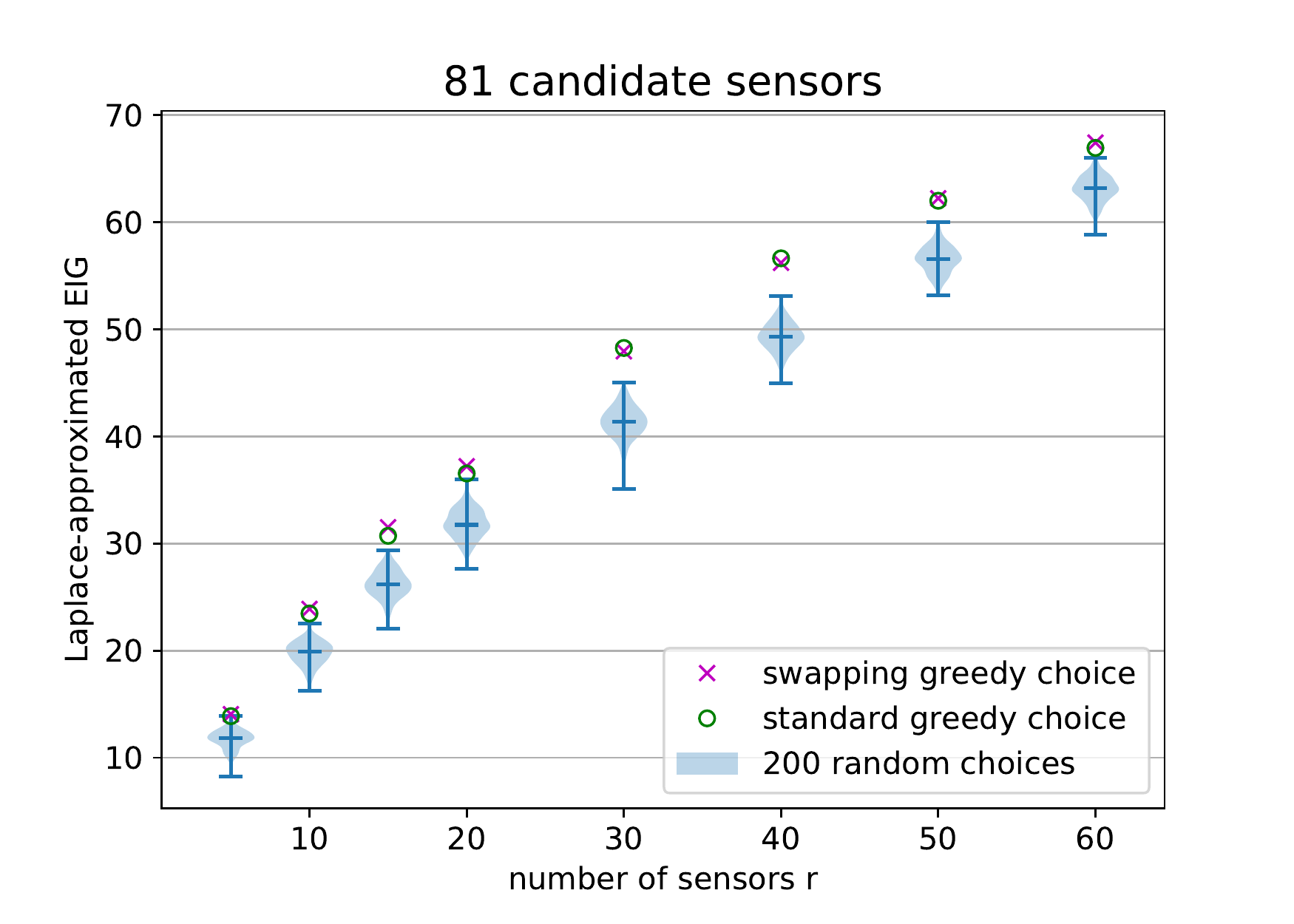}   
  \end{subfigure}

\caption{Laplace approximate EIG with increasing number of sensors out of 9 (left) and 81 (right) candidates.  Blue filled area represents the probability of $\hat{\Psi}$ for all the designs with lines at the minimum, maximum and median.} 
\label{fig:poi}

\end{figure}

\begin{table}[!htbp]
\caption{Approximate EIG $\hat{\Psi}$ at the design chosen by the standard greedy algorithm, the swapping greedy algorithm and at the optimal design, with the ranking of greedy choices among all the possible designs in the bracket. 
}
\begin{center}
\begin{tabular}{|c|c|c|c|c|c|c|c|}
\hline
 $r$ &$2$ &$3$&$4$  & $5$&$6$&$7$ & $8$\\
\hline
\tabincell{c}{standard\\greedy} &$6.5563(\#1)$ &$8.8113(\#1)$&$10.5916(\#2)$  & $12.4633(\#1)$&$13.8670(\#3)$&$15.2979(\#2)$ & $16.3798(\#1)$ \\
\hline
\tabincell{c}{swapping\\greedy} &$6.5563(\#1)$ &$8.8113(\#1)$&$10.5916(\#2)$  & $12.4633(\#1)$&$13.8670(\#3)$&$15.2979(\#2)$ & $16.3798(\#1)$\\
\hline
\tabincell{c}{optimal\\choice} &$6.5563$ &$8.8113$&$10.7776$  & $12.4633$&$14.1259$&$15.3501$ & $16.3798$\\
\hline
\end{tabular}
\end{center}
\label{table:linear}
\end{table}

Then we consider $81$ candidate sensor locations shown in Figure \ref{fig:1} (right). We randomly draw $200$ different designs from the candidate sensors and compute their LA-EIG, and compare them with LA-EIG of the choices by two greedy algorithms. Figure \ref{fig:poi} (right) illustrates that the design chosen by both greedy algorithms are much better than all the random choices, and the swapping greedy algorithm is mostly better or equal to the standard greedy algorithm. 
To illustrate the stability of our method, we compare the greedy choices with $200$ random designs with increasing number of training data $N_s$ and increasing number of candidate sensor locations $d$ (data dimension) in Figure \ref{fig:2}. We can see that the swapping greedy algorithm can always find better designs than the standard greedy algorithm, and much better than the random choices. 
\begin{figure}[h!]
  \centering
  \begin{subfigure}[b]{0.45\linewidth}
    \includegraphics[width=\linewidth]{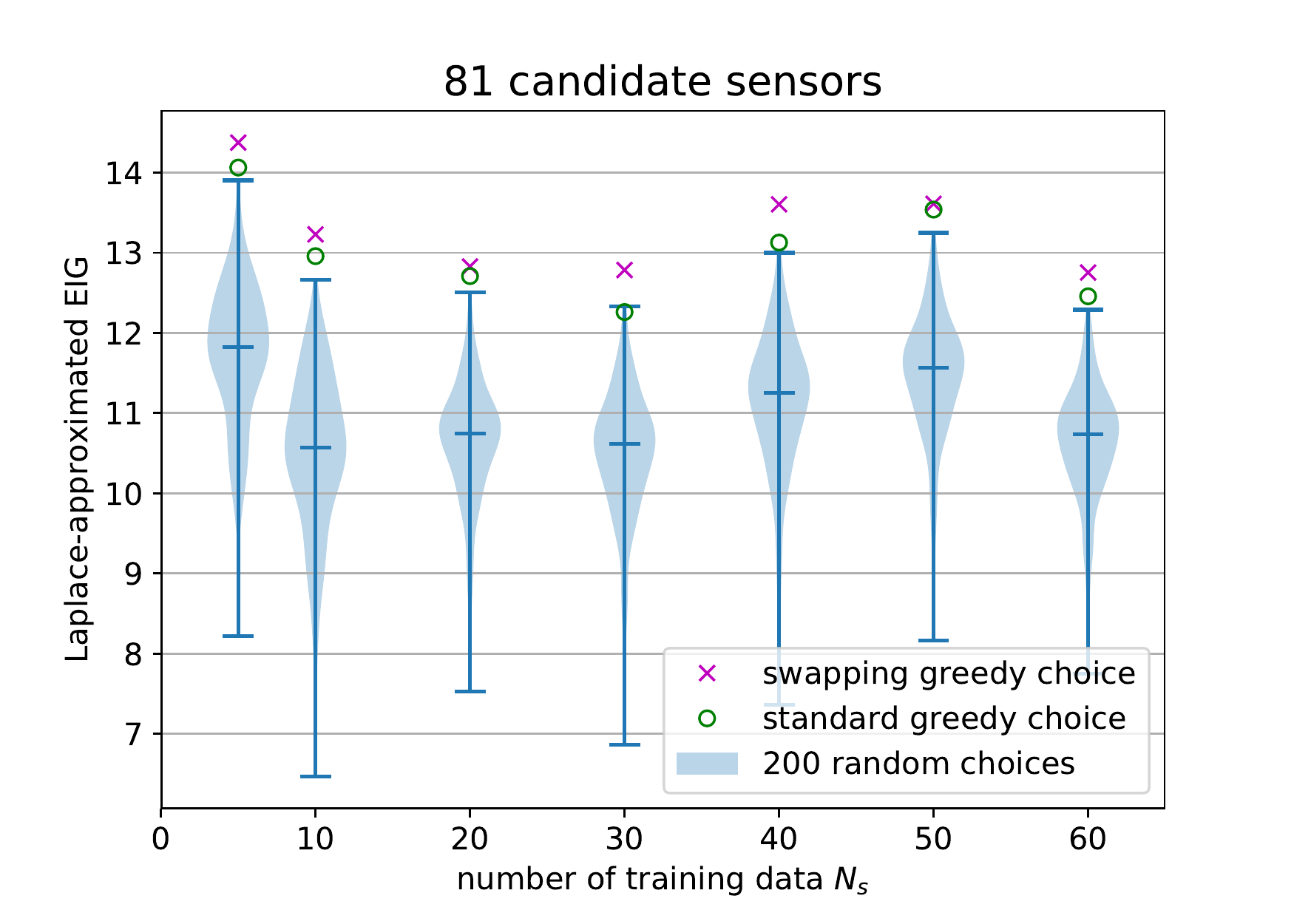}   
  \end{subfigure}
  \begin{subfigure}[b]{0.45\linewidth}
    \includegraphics[width=\linewidth]{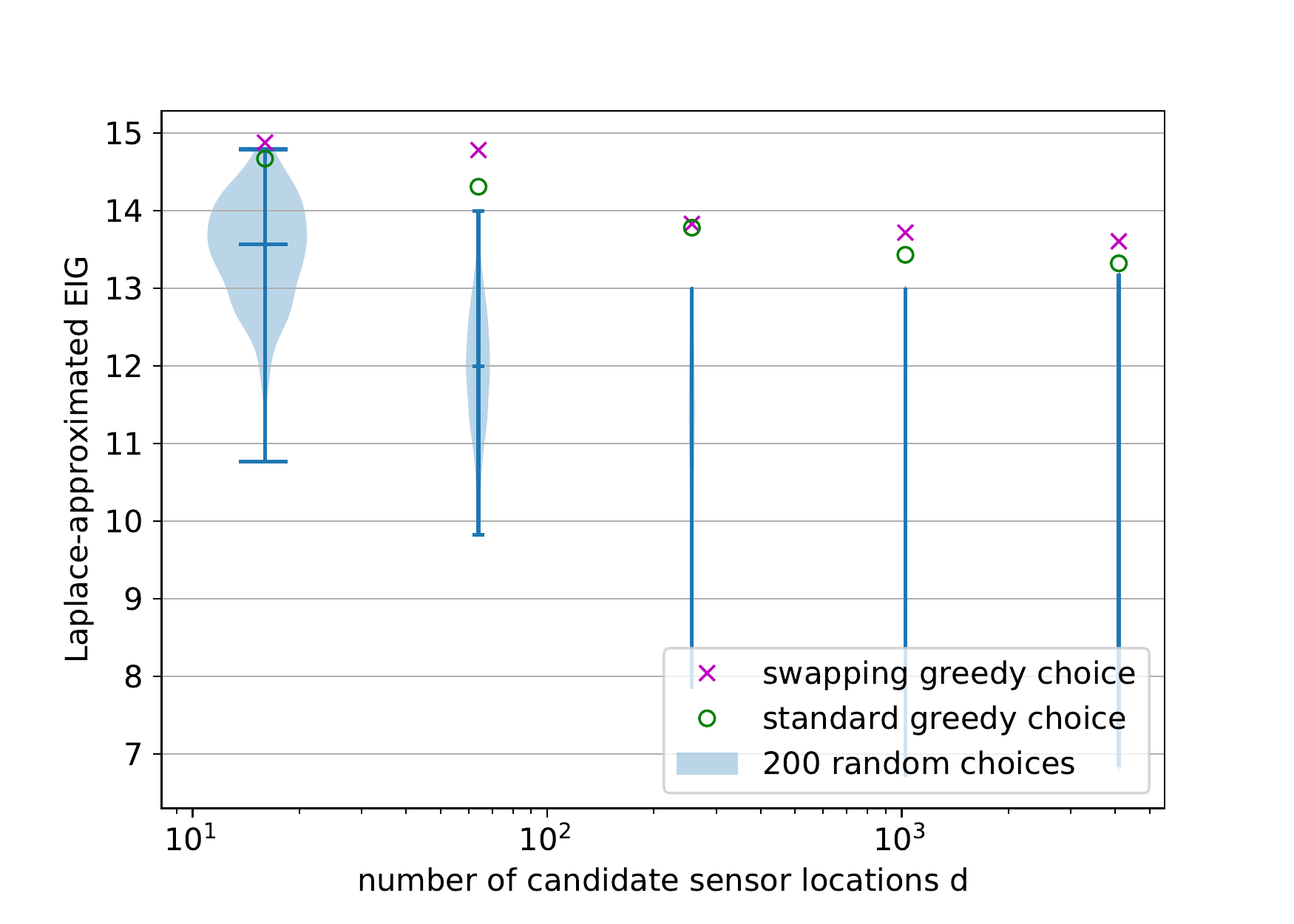}   
  \end{subfigure}

\caption{LA-EIG for choosing 5 sensors at optimal (by the swapping and standard greedy algorithms) and random designs with increasing number of training data (left) and increasing number of candidates (right). Blue filled area represents the probability of $\hat{\Psi}$ for $200$ random designs. 
} 
\label{fig:2}
\end{figure}


\subsubsection{Scalability}
As analyzed in Section \ref{sec:complexity}, the computational complexity in terms of the number of PDE solves critically depends on the rank $r$ in the low-rank approximation of the Hessian $\hat{\mathbf{H}}_d$. We investigate its dependence on the parameter dimension and data dimension (the number of candidate sensor locations). The decay of the eigenvalues of the Hessian are shown in Figure \ref{fig:scalability} with increasing parameter dimension and data dimension. From the similar decay of the eigenvalues in the left part of the figure, we can conclude that our algorithm is (strongly) scalable with respect to the parameter dimension in achieving the same information with the same number of sensors, while from the similar decay rate of the eigenvalues in the right part of the figure, we conclude that it is (weekly) scalable with respect to the data dimension in achieving the sample percentage of information with the same number of sensors.

\begin{figure}[h!]
  \centering
  \begin{subfigure}[b]{0.45\linewidth}
    \includegraphics[width=\linewidth]{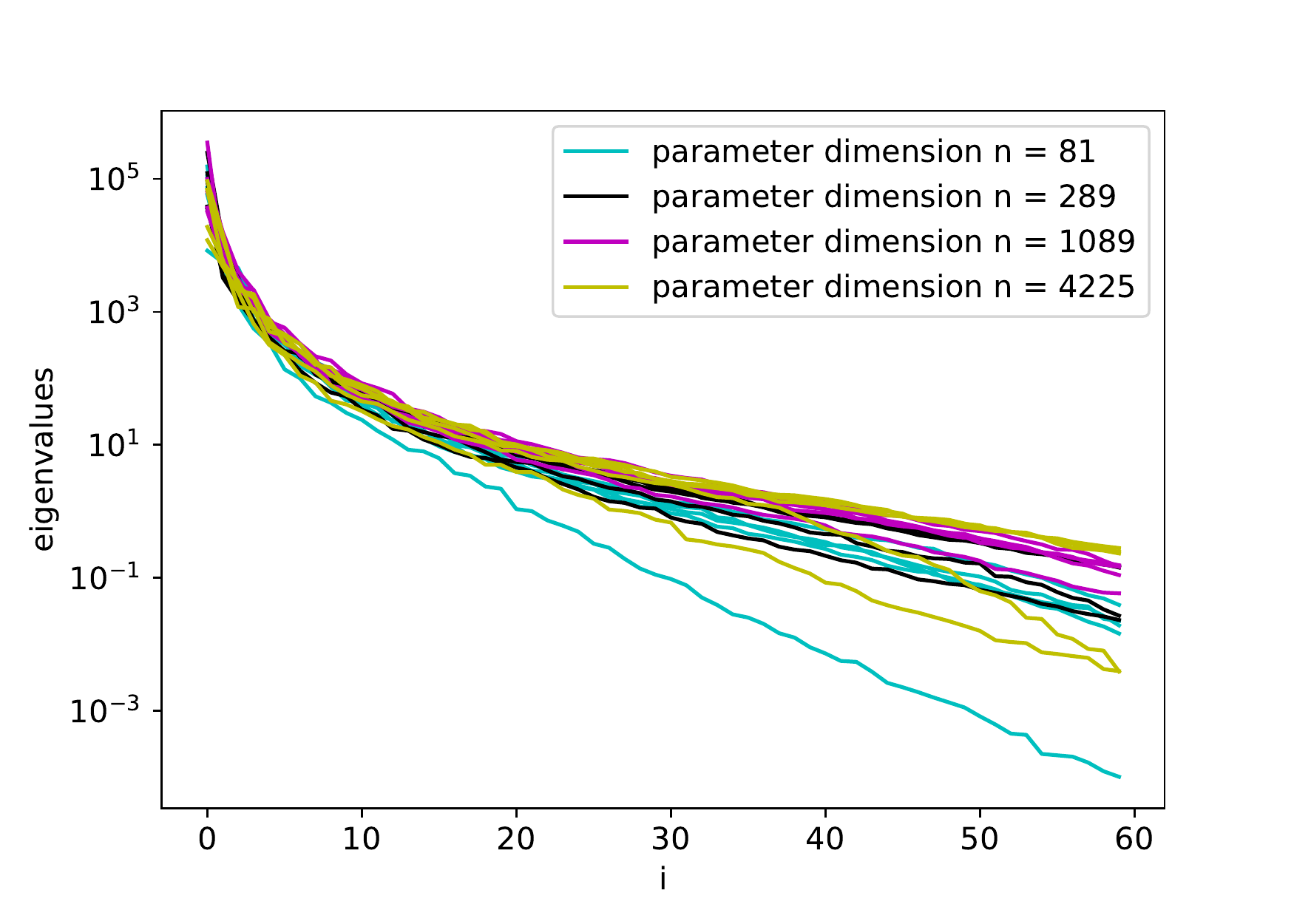}   
  \end{subfigure}
  \begin{subfigure}[b]{0.45\linewidth}
    \includegraphics[width=\linewidth]{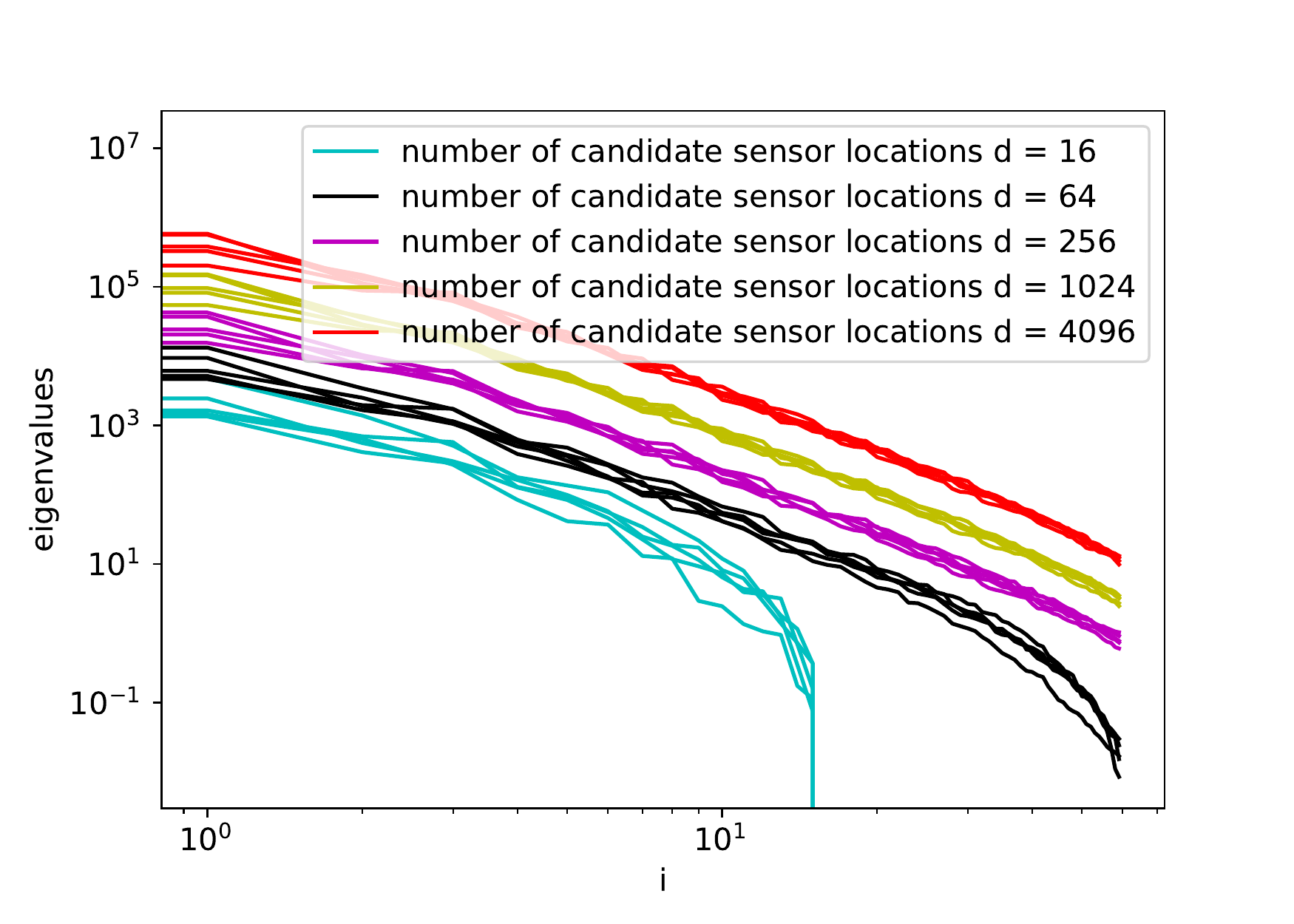}   
  \end{subfigure}
\caption{Left: eigenvalues with increasing parameter dimension. Right: eigenvalues with increasing number of candidate sensor locations. Multiple lines with the same color represent different training data cases.} 
\label{fig:scalability}
\end{figure}

\section{Conclusion}
\label{sec:conclusion}

We have developed a computational framework for optimal experimental design of infinite-dimensional Bayesian inverse problem governed by PDE. Our method exploits low-rank structure of the prior-preconditioned data misfit Hessian, dominant data subspace information from the Jacobian of parameter-to-observable map, laplace approximation of nonlinear Bayesian posterior. We reduce the computational cost of number of PDE solves to a limited number of PDE solves offline. The numerical tests of linear advection-diffusion problem and nonlinear poison problem illustrate the effectiveness and scalability of our methods.  Our limitation lies on the assumption of Gaussian prior and Gaussian additive noise, the series of approximations for nonlinear problems in sacrifice for cheap computation, and the nature of greedy algorithm. Some future work remains: (i) extension to nonlinear problems that laplace method cannot provide a good approximation such as helmholtz equations. (ii)  theoretical guarantees of swapping greedy algorithms. (iii) adaptively determine the number of sensor needed rather than a fixed number.

\bibliographystyle{plain}
\bibliography{paper}

\end{document}